\documentclass[reqno]{amsart}

\copyrightinfo{2010}{American Mathematical Society}

\usepackage{amsfonts}
\usepackage{amssymb}

\usepackage{color}
\usepackage[dvipsnames]{xcolor}
\usepackage{hyperref}

\newtheorem{theorem}{Theorem}[section]
\newtheorem{lemma}[theorem]{Lemma}
\newtheorem{corollary}[theorem]{Corollary}

\theoremstyle{definition}

\newtheorem{example}[theorem]{Example}

\theoremstyle{remark}
\newtheorem{remark}[theorem]{Remark}

\numberwithin{equation}{section}

\newcommand{\cC}{\mathcal{C}}
\newcommand{\cD}{\mathcal{D}}
\newcommand{\iD}{\mathit{\varDelta}}

\newcommand{\Id}{I}

\newcommand{\R}{\mathbb{R}}
\newcommand{\bs}{\mathbf{s}}
\newcommand{\cX}{\mathcal{X}}
\newcommand{\Z}{\mathbb{Z}}

\begin{document}

\title[On a Singular Perturbation of the Navier-Stokes Equations]
      {On a Singular Perturbation \\ of the Navier-Stokes Equations}

\author[A. Shlapunov]{Alexander Shlapunov}

\address{Siberian Federal University,
         Institute of Mathematics and Computer Science,
         pr. Svobodnyi 79,
         660041 Krasnoyarsk,
         Russia}

\email{ashlapunov@sfu-kras.ru}


\author[N. Tarkhanov]{Nikolai Tarkhanov}

\address{Institute of Mathematics,
         University of Potsdam,
         Karl-Liebknecht-Str. 24/25,
         14476 Potsdam OT Golm,
         Germany}

\email{tarkhanov@math.uni-potsdam.de}


\date{February 23, 2018}




\begin{abstract}
The paper is aimed at analysing a singular perturbation of the Navier-Stokes
equations on a compact closed manifold.
The case of compact smooth manifolds with boundary under the Dirichlet conditions
is also included.
Global existence and uniqueness is established for the weak solutions of the Cauchy
problem.
The solution of the regularised system is shown to converge to the solution of the
conventional Navier-Stokes equations provided it is uniformly bounded in parameter.
\end{abstract}

\maketitle

\tableofcontents

\section*{Introduction}
\label{s.Introdu}

The classical Navier-Stokes equations for the incompressible viscous fluid in a
domain $\cX$ of $\R^n$ have the form
\begin{equation}
\label{eq.cNSefivf}
\begin{array}{rcl}
   \partial_t v - \nu \iD v + N (v) + \mathrm{grad}\, p
 & =
 & f,
\\
   \mathrm{div}\, v
 & =
 & 0
 \end{array}
\end{equation}
for $(x,t) \in \cX \times I$, where $I = (0,T)$ is finite or infinite time interval.
Here,
   $v$ stands for the velocity vector of the flow,
   $p$ for the pressure and
   $f$ for the outer force field,
all of them being functions of $(x,t)$.
The quantity $\nu$ is the so-called kinematic viscosity of the fluid and $\nu$ is
assumed to be a positive constant.
The resistance of system (\ref{eq.cNSefivf}) to mathematical study is caused by the
nonlinear term $N (v) = v'_x v$, where $v'_x$ is the Jacobi matrix of the mapping
$\cX \to \R^n$ given by $x \mapsto v (x,t)$ at any instant $t \in I$.
Being a bidifferential operator, $N$ fails to be Lipschitz continuous on unbounded
subsets of fit function spaces.
There is a huge literature devoted to the Navier-Stokes equations, for which we refer
the reader to
   \cite{Tema79},
   \cite{Lady03} and
   \cite{Tao16}.

As the Navier-Stokes equations are nonlinear, one needs a specific trick to introduce
the concept of a weak solution for them.
To wit, a locally square integrable vector $v$ on $\cX \times I$ is called a weak
solution of (\ref{eq.cNSefivf}) if it is weakly divergence-free and if
\begin{equation}
\label{eq.weaksol}
   \int \!\!\!\! \int_{\cX \times I}
   \left( - (v, \partial_t w)_{x,t} - \nu\, (v, \iD w)_{x,t} - (v, w'_x v)_{x,t}
   \right)
   dx dt
 = \int \!\!\!\! \int_{\cX \times I}
   (f,w)_{x,t}
   dx dt
\end{equation}
for all smooth vector fields $w$ with compact support in $\cX \times I$ satisfying
   $\mathrm{div}\, w = 0$.
This definition is due to \cite{Hopf51} who
   specified it for solutions of the first mixed problem for the Navier-Stokes system
   in the cylinder $\cX \times I$ with initial data at $t = 0$ and zero Dirichlet data
   on the lateral surface $\partial \cX \times I$
and
   established the existence of a weak solution to the first mixed problem with finite
   energy norm.
In
   \cite{Lady70},
   \cite{Lion69} and
   \cite{Tema79}
one uses a slightly different concept of a weak solution which refines that of \cite{Hopf51}
in the case where the weak solution belongs to
   $H^1 (I, L^2 (\cX,\R^n)) \cap L^2 (I, H^1_0 (\cX,\R^n))$.
It does not apply to define weak solutions in $L^2_{\mathrm{loc}} (\cX \times I,\R^n)$.
Note that, under finite energy norm, the uniqueness property for the classical Navier-Stokes
equations takes place first in
   $L^r (I, L^q (\cX,\R^n))$
with $q > n$ and $2/r + n/q \leq 1$,
   see \textit{ibid}.

Suppose $f = 0$.
Then a direct calculation shows that $v = c (t) \mathrm{grad}\, h (x)$ is a solution of
(\ref{eq.weaksol}), provided only that $c$ is a square integrable function and $h$ is
harmonic.
The vector field $v$ is infinitely differentiable so far as the space variables are
concerned but need not possess any time derivatives whatever.
The conclusion which we draw from this example is that the order of time differentiability
of a weak solution is intimately tied to the amount of time regularity which is originally
assumed.
On the other hand, quite modest assumptions concerning the spatial regularity of a weak
solution are enough to guarantee $C^\infty$ with respect to the space variables,
   see \cite{Serr62}.

As usual the velosity $v$ is subject to an initial condition $v = v_0$ on the lower base of
$\cX \times I$ and a Dirichlet condition $v = v_l$ on the lateral boundary of the cylinder.
If one wants to work in a framework where the boundary conditions are understood in the sense
of traces, and not just weakly, then one can choose an $L^q$ setting and require that the
solution $v$ belongs to the space
   $W^{1,q} (I, L^q (\cX,\R^n)) \cap L^q (I, W^{2,q} (\cX,\R^n))$
for some $q \in (1,\infty)$.
In order to obtain the optimal spaces for the data, one needs sharp trace results for the
space.
It is known that the trace $v_0$ of $v$ at $t = 0$ belongs to
   $W^{2 - 2/q,q} (\cX,\R^n)$
and the restriction $v_l$ of $v$ to the lateral boundary belongs to the anisotropic space
   $W^{1-1/2q,q} (I, L^q (\partial \cX,\R^n)) \cap L^q (I, W^{2-1/q,q} (\partial \cX,\R^n))$,
cf. \cite{Slob58}.
Observe that the restriction to the lateral boundary still retains some time regularity.
Moreover, one has
   $f \in L^q (I, L^q (\cX,\R^n))$,
the latter space being $L^q (\cX \times I,\R^n)$, and the compatibility condition $v_0 = v_l$
should hold at $t = 0$, if $v_l$ has a trace at $t = 0$, which happens for $q > 3$.
One often takes a large $q$ to simplify the treatment of the nonlinearity.
Thus, the norm of $W^{2 - 2/q} (\cX,\R^n)$ is far away from the norms one can control by the
usual a priori estimates, such as the norms of $H^1$ or $L^\infty$.

Since \cite{Hopf51} much efforts were made to prove the uniqueness and regularity of the
weak solution to the first mixed problem for the Navier-Stokes equations.
However, there has been no breakthrough unless $n = 2$, in which case the results are due
to \cite{Lera34a,Lera34b} and independently \cite{Lady70}.
In \cite{Lady03} one finds certain indirect evidence to the fact that the weak solution of
\cite{Hopf51} might be not unique, and so fail to be regular.
The paper \cite{Tao16} establishes a finite time blowup for an averaged three-dimensional
Navier-Stokes equation.
Our starting point in \cite{ShlaTark16} is to give the Navier-Stokes equations a regular
domain to have got the uniqueness from the very beginning.
We proved in \cite{ShlaTark16} that the Cauchy problem for (\ref{eq.cNSefivf}) in all of
$\R^n$ is stable in weighted H\"{o}lder spaces, i.e., the corresponding mapping is open.
On the other hand, the proof of the range closedness requires hard a priori estimates which
we have not been able to prove.
They seem to be closely tied to the very specific form of the nonlinear term $v'_x v$ in
(\ref{eq.cNSefivf}).
As mentioned, on replacing $v'_x v$ away from a ball of any radius by a mapping which is
Lipschitz continuous we immediately get existence.
To understand to what extent it is physically meaningful to perturb the nonlinear term we
introduced in \cite{MeraShlaTark15} generalised Navier-Stokes equations related to elliptic
complexes.

The first paper to consider the Navier-Stokes equations on Riemannian manifolds is the
classical paper \cite{EbinMars70}.
There, the authors employ the Hodge Laplacian but add a note in the proof emphasizing
that the ``correct'' Laplacian for Navier-Stokes should be constructed from the Bochner
Laplacian, referring to \cite{Serr59}.
This point is reiterated by \cite{Tayl10} and taken up again in
   \cite{ChanCzub13} and
   \cite{Lich16},
the last papers dealing precisely with the issue of non-uniqueness for the Navier-Stokes
equations on manifolds.

In this paper we analyse a regularised form of the conventional Navier-Stokes equations.
More precisely, we consider the first mixed problem for solutions of the system
\begin{equation}
\label{eq.gNSefivf}
\begin{array}{rcl}
   \partial_t v + \varepsilon (- \iD)^m v - \nu \iD v + N (v) + \mathrm{grad}\, p
 & =
 & f,
\\
   \mathrm{div}\, v
 & =
 & 0
 \end{array}
\end{equation}
in the cylinder $\cX \times I$, where
   $m$ is a natural number satisfying $2 m - 1 \geq n/2$
and
   $\varepsilon$ a small positive parameter.
Given any data
   $f \in L^2 (\cX \times I,\R^n)$ and
   $v_0 \in H^m_0 (\cX,\R^n)$,
we show the existence of a weak solution
   $v \in L^\infty (I, L^2 (\cX,\R^n)) \cap L^2 (I, H^m_0 (\cX,\R^n))$
to the problem with zero data on the lateral surface of the cylinder.
Moreover, the solution fulfills an energy estimate, it is unique and inherits the regularity
from $f$ and $u_0$.

If the solution $u$ is bounded in $L^2 (I, H^m_0 (\cX,\R^n))$ uniformly in $\varepsilon$,
then it approaches the unique weak solution to the conventional problem in the norm of
   $L^\infty (I,L^2 (\cX,\R^n)) \cap L^2 (I,H^1_0 (\cX,\R^n))$.
The conclusion which we might draw from these results is that system (\ref{eq.gNSefivf})
with $m \geq (n+2)/4$ is an adequate regularisation of the equations of dynamics of
incompressible viscous fluid in an $n\,$-dimensional space.
We focus on the case where $\cX$ is a compact $C^\infty$ manifold without boundary.
The case of compact smooth manifolds with boundary under the zero Dirichlet conditions is
also included.

\part{The generalised Navier-Stokes equations}
\label{p.tgNSequ}

\section{The Navier-Stokes equations for elliptic complexes}
\label{s.tNSefec}

Let $\cX$ be a compact closed Riemannian manifold of dimension $n$.
Consider an elliptic complex
\begin{equation}
\label{eq.ellcomp}
   0
 \longrightarrow
   C^\infty (\cX,F^0)
 \stackrel{A^0}{\longrightarrow}
   C^\infty (\cX,F^1)
 \stackrel{A^1}{\longrightarrow}
   \ldots
 \stackrel{A^{N-1}}{\longrightarrow}
   C^\infty (\cX,F^N)
 \longrightarrow
   0
\end{equation}
of first order differential operators $A^i$ on $\cX$ acting between sections of smooth
vector bundles $F^i$ over $\cX$.
As usual, we write $A$ for the graded operator associated with sequence (\ref{eq.ellcomp}),
which is defined by $Au = A^i u$ for sections $u$ of $F^i$.
The operator $A$ is called the differential of complex (\ref{eq.ellcomp}) and its square
$A^2$ vanishes, for the composition of any two neighbouring operators in (\ref{eq.ellcomp})
is zero by the definition of a complex.

\begin{remark}
\label{r.copsdo}
Slight changes in the proofs actually show that the theory still applies to complexes of
pseudodifferential operators.
But it is not our purpose to develop this point here.
\end{remark}

Choosing a volume form $dx$ on $\cX$ and a Riemannian metric $(\cdot,\cdot)_x$ in the fibres
of $F^i$, we equip each bundle $F^i$ with a smooth bundle homomorphism
   $\ast : F^i \to F^i{}^\ast$
defined by
   $\langle \ast u, v \rangle = (v,u)_x$ for $u, v \in F^i_x$,
and the space $C^\infty (\cX,F^i)$ with the unitary structure
$$
   (u,v) = \int_{\cX} (u,v)_x dx
$$
giving rise to the Hilbert space $L^2 (\cX,F^i)$.
Write $A^i{}^\ast$ for the formal adjoint of $A^i$ and $A^\ast$ for the corresponding graded
operator given by
   $A^\ast u = A^{i-1}{}^\ast u$
for sections $u$ of $F^i$.
By analogy with the terminology used in hydrodynamics,
   we call divergence-free those sections $u$ of $F^i$ which satisfy $A^\ast u = 0$,
   by the vorticity of $u$ is meant the section $Au$,
and
   any section $\pi$ of $F^{i-1}$ satisfying $A \pi = u$ is said to be a potential of $u$.

The ellipticity of complex (\ref{eq.ellcomp}) proves to be equivalent to saying that the
Laplacian
   $\iD = A^\ast A + A A^\ast$
is an elliptic pseudodifferential operator of order two on sections of $F^i$, for every
$i = 0, 1, \ldots, N$, see \cite{Tark95}.
The Laplacian $\iD$ is nonnegative on $C^\infty (\cX,F^i)$, for
   $(\iD u, u) = \| Au \|^2 + \| A^\ast u \|^2$
for all sections $u \in C^\infty (\cX,F^i)$.
In the sequel we tacitly assume that $i$ is fixed, thus restricting our attention to complex
(\ref{eq.ellcomp}) at step $i$.

Let $m$ be a nonnegative integer.
Our standing assumption on $m$ will be that $m \geq (n+2)/4$.
The system
\begin{equation}
\label{eq.mhdefivf}
\begin{array}{rcl}
   \partial_t v + \varepsilon \iD^m v + \nu \iD v + N (v) + Ap
 & =
 & f,
\\
   A^\ast v
 & =
 & 0
 \end{array}
\end{equation}
in the cylinder $\cC = \cX \times I$ for the unknown sections $v$ and $p$ of $F^i$ and $F^{i-1}$,
respectively, is an obvious reformulation of (\ref{eq.gNSefivf}) within the framework of elliptic
complexes.
If (\ref{eq.ellcomp}) is the de Rham complex on $\cX$ and $i = 1$, then (\ref{eq.mhdefivf}) just
amounts to (\ref{eq.gNSefivf}) up to the duality between one-forms and vector fields on $\cX$.
In the general case, by $N$ is meant any nonlinear cochain mapping of complex (\ref{eq.ellcomp})
given by a sequence of first order bidifferential operators $N^i (u,v)$ on $\cX$, and $N^i (u)$
stands for $N^i (u,u)$.

\begin{remark}
\label{eq.mietn/2}
One can take specifically $m = (n+2)/4$ but we did not so to avoid considering the fractional powers
of pseudodifferential operators, see \cite{Shub87} and elsewhere.
\end{remark}

Note that system (\ref{eq.mhdefivf}) is equivalent to the generalised Navier-Stokes equations in
the sense of \cite{MeraShlaTark15} corresponding to the elliptic complex of (\ref{eq.ellcomp}) whose
differential is changed to
   $\sqrt{\Id + (\varepsilon/\nu) \iD^{m-1}} A$.

If $s \in \Z_{\geq 0}$, we use the designation $\bs (s)$ for the pair $(2ms,s)$.
Let $H^{\bs (s)} (\cC,F^i)$ stand for the Hilbert space of all sections of the bundle $F^i$ over
$\cC$ of Slobodetskij class $H^{\bs (s)}$, see \cite{Slob58}.
A section $u$ is said to belong to $H^{\bs (s)} (\cC,F^i)$ if,
   for any coordinate patch $U$ in $\cX$ over which the bundle $F^i$ is trivial,
the derivatives $\partial_x^\alpha \partial_t^j u$ are locally square integrable in
   $U \times \overline{I}$
whenever $|\alpha| + 2mj \leq 2ms$.
On choosing a finite covering of $\cX$ by such patches and a $C^\infty$ partition of unity on $\cX$
subordinate to the covering we introduce a scalar product in $H^{\bs (s)} (\cC,F^i)$.
The corresponding norm is independent on the choice of coordinates and local trivialisations up to
the norm equivalence.
It is well known that each section of $H^{\bs (s)} (\cC,F^i)$ possesses a trace on the base of the
cylinder which belongs to
   $H^{2m (s-1/2)} (\cX,F^i)$,
the result being due to \cite{Slob58}.

To specify a solution $v$ to system (\ref{eq.mhdefivf}) we prescribe initial values at $t = 0$ to
$v$.
To wit,
\begin{equation}
\label{eq.initval}
   v (\cdot,0) = v_0
\end{equation}
on $\cX$, where $v_0$ is a given section of $F^i$.
We will look for a velocity in $H^{\bs (1)} (\cC,F^i)$, hence $v_0$ is assumed to be in
   $H^{m} (\cX,F^i)$
unless otherwise stated.
One imposes no initial condition on the pressure, for (\ref{eq.mhdefivf}) does not contain any time
derivative of $p$.
Equations (\ref{eq.mhdefivf}) supplemented with initial condition (\ref{eq.initval}) constitute the
Cauchy problem for the generalised Navier-Stokes equations.
For the existence of a solution to the Cauchy problem it is necessary that $A^\ast v_0 = 0$ on $\cX$,
which we will always assume.
While the velocity $v$ seems to be uniquely determined by the data $f$ and $v_0$, the pressure $p$
is determined only up to solutions of the homogeneous equation $Ap = 0$ on $\cX$.
The non-uniqueness of $p$ is evaluated by the dimension of the cohomology of complex (\ref{eq.ellcomp})
at step $i-1$.
Since the complex is elliptic, the cohomology is finite dimensional,
   see \cite{Tark95}.
We make use of this fact in the next section.

\section{Elimination of the pressure}
\label{s.eliotpr}

The approach of \cite{ShlaTark16} which consists in applying the differential $A$ to the both sides
of (\ref{eq.mhdefivf}) and passing to the equation for the vorticity of $v$ requires the cohomology
of (\ref{eq.ellcomp}) to be zero at step $i$.
If this is the case, then the Hodge decomposition theorem for elliptic complexes on compact closed
manifolds implies readily
$
   v = A^\ast G (Av)
$
for all divergence-free sections $v$ of $F^i$, where $G$ is the so-called Green operator related to
the Laplacian $\iD$,
   see \cite{Tark95}.
Hence it follows that
$
   A N (v) = A N A^\ast G (Av),
$
and so we can choose $N^{i+1} = A N^{i} A^\ast G$ as the $(N\!+\!1)\,$-th component of a new cochain
mapping of complex (\ref{eq.ellcomp}).
This mapping satisfies the commutativity relation $A N (v) = N (Av)$ for all divergence-free sections
$v$ of $F^i$ and bears moreover the property that
   $(N (w),z) = 0$ for all divergence-free sections $z$ of $F^{i+1}$.
Since pseudodifferential operators are involved in the very formulation of (\ref{eq.mhdefivf}), we
focus on the classical approach which expoits the Helmholtz projection, cf.
   \cite{MeraShlaTark15}.

The idea is to project the first equation of (\ref{eq.mhdefivf}) onto the space of divergence-free
sections of $F^i$.
Denote by $H$ the orthogonal projection onto the null-space of $\iD$ in $L^2 (\cX,F^i)$.

\begin{lemma}
\label{l.Helmholtz1}
For $H$ and $G$ defined above, the operator
$
   Pv := Hv + A^\ast A\, Gv
$
is an orthogonal projection in $L^2 (\cX,F^i)$.
\end{lemma}

\begin{proof}
Since the Laplacian is elliptic, the space of all $h \in L^2 (\cX,F^i)$ satisfying  $\iD h = 0$ in
$\cX$  is finite dimensional.
The elements of this space are called harmonic sections and they actually satisfy both
   $Ah = 0$ and
   $A^\ast h = 0$
in $\cX$.
The harmonic sections are $C^\infty$ sections of $F^i$, and so the orthogonal projection $H$ onto the
space is a smoothing operator.
Moreover, the Green operator is a formally selfadjoint pseudodifferential operator of order $-2$ on
sections of $F^i$, such that
   $H G = 0$ and
   $A G = G A$,
and
   the identity operator on sections of $F^i$ splits into $H + A^\ast A G + A A^\ast G$.
The decomposition $v = Hv + A^\ast A Gv + A A^\ast Gv$ valid for all $v \in L^2 (\cX,F^i)$ is
usually referred to as the generalised Hodge decomposition.
Since $A^2 = 0$, the summands are pairwise orthogonal, and so both
   $A^\ast A G$ and
   $A A^\ast G$
are orthogonal projections, too.
For a thorough discussion of Hodge theory we refer the reader to
   \cite[Ch.~4]{Tark95}.
\end{proof}

The projector $P$ is an analogue of the Helmholtz projector onto vector fields which are
divergence-free.
A slightly different approach to this decomposition is presented in \cite{Lady70}.

\begin{lemma}
\label{l.Helmholtz2}
In order that $Pv = v$ be valid it is necessary and sufficient that $A^\ast v = 0$ in $\cX$.
\end{lemma}

\begin{proof}
Suppose that $Pv = v$.
Then $A^\ast v = A^\ast (Hv + A^\ast A\, Gv)$ vanishes in $\cX$.
On the other hand, if $A^\ast v = 0$ in $\cX$, then $A^\ast (Gv) = 0$ whence $Pv = v$, as
desired.
\end{proof}

From Lemma \ref{l.Helmholtz2} it follows that $P$ vanishes on sections of the form $Ap$
with
   $p \in L^2 (\cX,F^{i-1})$ and
   $Ap \in L^2 (\cX,F^i)$.
Indeed, to prove this it suffices to show that $Ap$ is orthogonal to all sections $v$
satisfying $A^\ast v = 0$ in $\cX$.
For such a section $v$ we get
\begin{eqnarray*}
   (Ap, v)_{L^2 (\cX,F^i)}
 & = &
   (p, A^\ast v)_{L^2 (\cX,F^{i-1})}
\\
 & = &
   0,
\end{eqnarray*}
as desired.

On applying the Helmholtz projector to the first equation of system (\ref{eq.mhdefivf}) one
obtains
$$
   (Pv)'_t + \varepsilon\, P (\iD^m v) + \nu\, P (\iD v) + P (Ap) + P N (v) = Pf
$$
while the second equation of (\ref{eq.mhdefivf}) means that $Pv = v$ in $\cC$.
Since $P (Ap) = 0$, this allows one to eliminate the pressure from the first equation,
thus obtaining an equivalent form
$$
\begin{array}{rcl}
   (Pv)'_t + \varepsilon\, P \iD^m (Pv) + \nu\, P (\iD v) + P N (Pv)
 & = &
   Pf,
\\
   (\Id\!-\!P) Ap + (\Id\!-\!P) N (Pv)
 & = &
   (\Id-P) f
\end{array}
$$
of equations (\ref{eq.mhdefivf}), for
   $(\Id\!-\!P) P = 0$
and
\begin{eqnarray*}
   (\Id - P) \iD^m P
 & = &
   (\Id - H - A^\ast A G)\, \iD^m\, (H + A^\ast A G)
\\
 & = &
   A A^\ast G \left( A^\ast A \right)^{m+1} G
\\
 & = &
   0,
\end{eqnarray*}
the last equality being due to the fact that $A^\ast G = G A^\ast$.

In other words, we separate the Cauchy problem (\ref{eq.mhdefivf}), (\ref{eq.initval}) into
two single problems
$$
   (Pv)'_t + \varepsilon\, P \iD^m (Pv) + \nu\, P (\iD v) + P N (Pv)
 = Pf
$$
in $\cC$ under the initial condition
$$
   Pv\, (\cdot,0) = v_0
$$
on $\cX$ and
\begin{equation}
\label{eq.sefp}
\begin{array}{rclcl}
   Ap
 & =
 & (\Id - P)\, (f - N (Pv))
 & \mbox{in}
 & \cC,
\\
   (p,u)_{L^2 (\cX,F^{i-1})}
 & =
 & 0
 & \mbox{for}
 & u \in \ker A.
\end{array}
\end{equation}

The operator $P \iD^m$ is sometimes called the Stokes operator.
It is a pseudodifferential operator of order $2m$ on sections of $F^i$.
The only solution of problem (\ref{eq.sefp}) is given by
$$
   p = A^\ast G (\Id - P)\, (f - N (Pv)).
$$

\begin{lemma}
\label{l.Stokes}
Suppose that $u \in H^{\bs (1)} (\cC,F^i)$ is a solution of the Cauchy problem
\begin{equation}
\label{eq.Stokes}
\begin{array}{rclll}
   u'_t + \varepsilon\, P \iD^m u + \nu\, P (\iD u) + P N (u)
 & =
 & Pf
 & \mbox{in}
 & \cC,
\\
   u (\cdot,0)
 & =
 & v_0
 & \mbox{on}
 & \cX
\end{array}
\end{equation}
in the cylinder.
Then $A^\ast u = 0$ in $\cC$.
\end{lemma}

\begin{proof}
Indeed, from the differential equation of (\ref{eq.Stokes}) it follows that
$$
   \frac{\partial}{\partial t}\, A^{\ast} u = 0
$$
in $\cC$.
Since
   $A^\ast u = A^\ast v_0 = 0$
for $t = 0$, we deduce readily that $A^\ast u = 0$ for all $t \in I$, as desired.
\end{proof}

On summarising what has been said we choose the following way of solving equations
(\ref{eq.mhdefivf}).
We first construct a solution $u$ of Cauchy problem (\ref{eq.Stokes}).
According to Lemma \ref{l.Stokes}, $u$ satisfies $A^\ast u = 0$ in $\cC$, and so
$Pu = u$.
Substitute the section $v := u$ into equation (\ref{eq.sefp}) for $p$.
From this equation the pressure $p$ is determined uniquely and bears the appropriate
regularity of the canonical solution of the inhomogeneous equation related to complex
(\ref{eq.ellcomp}) at step $i$.
Finally, on combining equations (\ref{eq.Stokes}) and (\ref{eq.sefp}) we conclude that
the pair $(v,p)$ is a solution of (\ref{eq.mhdefivf}) under condition (\ref{eq.initval}).

In the sequel we focus on the study of operator equation (\ref{eq.Stokes}) by Hilbert space
methods.

\section{Fundamental solution to the heat equation}
\label{s.fsttheq}

For sections of $F^i$, consider the equation of evolution
\begin{equation}
\label{eq.heatequ}
   \frac{\partial u}{\partial t}
 = - \varepsilon \iD^m u - \nu \iD u
\end{equation}
on the semiaxis $t > 0$.
It is easily seen that the endomorphism of $F^i_x$ given by the principal symbol
   $- \varepsilon \left( \sigma^2 (\iD) (x,\xi) \right)^m$
of the operator on the right-hand side of (\ref{eq.heatequ}) at any nonzero vector
$\xi \in T^\ast_x \cX$ has only real eigenvalues strictly less than zero.
Therefore, the equation (\ref{eq.heatequ}) is actually one of the well-known type of
partial differential equations, called parabolic, which enjoy a behaviour essentially
like that of the classical equation of the heat conduction.
For instance, one has as a smoothing operator the fundamental solution
$
   t \mapsto \psi^i (x,y,t) \in C^{\infty} (\cX \times \cX, F^i \otimes F^i{}^\ast)
$
which generates the solution of the Cauchy problem for (\ref{eq.heatequ}) in the sence
that, for any $u_0 \in C^{\infty} (\cX,F^i)$,
$$
   (\varPsi^i u_0) (x,t)
 = \int_{\cX} (u_0, \ast^{-1} \psi^i (x,\cdot,t))_{y} dy
$$
is regular for $t > 0$, satisfies (\ref{eq.heatequ}) and has initial value $u_0$.
The fundamental solution $\varPsi^i$ is here unique by virtue of the compactness of $\cX$.
We write $\varPsi$ for the corresponding graded operator.

We recall briefly the Hilbert-Levi procedure for constructing the fundamental solutions
for parabolic operators, and state certain basic estimate for them, only to the extent
which we shall need later.
For the details we refer the reader to \cite{Eide56} and elsewhere.

Denote by
$
   L = \partial_t + \varepsilon \iD^m + \nu \iD
$
the parabolic operator of (\ref{eq.heatequ}) on sections of the induced bundle $F^i$ over
$\cX \times \R_{> 0}$.
We can cover $\cX$ by a finite number of coordinate patches $U$ such that over each $U$
the bundle $F^i$ is trivial, i.e., the restriction $F^i \restriction_U$ is isomorphic to
the trivial bundle $U \times \mathbb{C}^{k_i}$, where $k_i$ is the fibre dimension of
$F^i$.
Let $x = (x^1, \ldots, x^n)$ be local coordinates in $U$.
Then any section of $F^i$ over $U$ can be regarded as a $k_i\,$-column of complex-valued
functions of coordinates $x$.
Under this identification, for any $u \in C^\infty (U,F^i)$, the $k_i\,$-column $\iD^m u$
of functions in $U$ is represented by
$$
   \iD^m u
 = \Big( \sum_{|\alpha| \leq 2} \iD_{\alpha} (x)\, \partial^{\alpha} \Big)^m u,
$$
where
   $\iD_{\alpha} (x)$ are $(k_i \times k_i)\,$-matrices of $C^\infty$ functions of $x$,
   by $\partial^{\alpha}$ is meant the derivative
   $(\partial/\partial x^1)^{\alpha_1} \ldots (\partial/\partial x^n)^{\alpha_n}$,
and
   the upper index $m$ refers to the power with respect to the composition of operators.
So, the principal symbol of $\iD^m$ is given by
$$
   \sigma^{2m} (\iD^m) (x,\xi)
 = \Big( - \sum_{|\alpha| = 2} \iD_{\alpha} (x)\, \xi^{\alpha} \Big)^m
$$
for $(x,\xi) \in U \times \R^n$, where
   $\xi^{\alpha} = \xi_1^{\alpha_1} \ldots \xi_n^{\alpha_n}$.
As already mentioned, $\sigma^{2m} (\iD^m) (x,\xi)$ has only strictly negative eigenvalues
for each nonzero vector $\xi$, and hence the matrix
\begin{equation}
\label{eq.paramet}
   p_U (x,y,t)
 = \frac{1}{(2 \pi)^n}
   \int
   \exp \left( t\, \varepsilon \sigma^{2m} (\iD^m) (x,\xi) + \sqrt{-1}\, \langle x-y, \xi \rangle\, \Id
        \right)
   d \xi
\end{equation}
is well defined for all $x, y \in U$ and $t > 0$, where
   $\Id$ stands for the unity matrix of type $k_i \times k_i$.

It will be seen below that $p_U$ describes the principal part of singularity for the fundamental
solution of the parabolic operator $L$ in $U$.
For that reason $p_U$ is said to be a local parametrix for this operator.
Letting now
$$
   \sum \varphi_U^2 \equiv 1
$$
be a quadratic partition of unity on $\cX$ subordinate to a covering $\{\, U\, \}$, and considering
the expression
$
   \varphi_U (x) p_U (x,y,t) \varphi_U (y)
$
as a section of $F^i \otimes (F^i)^\ast$ with support in $U \times U$, we define a global parametrix
for (\ref{eq.heatequ}) by
$$
   p (x,y,t)
 = \sum_U \varphi_U (x) p_U (x,y,t) \varphi_U (y).
$$
The fundamental solution $\psi (x,y,t)$ of $L$ will then be given as the unique solution of the
integral equation of Volterra type
$$
   \psi (x,y,t)
 = p (x,y,t)
 - \int_0^t dt' \int_{\cX} (L p (\cdot,y,t'), \ast^{-1} \psi (x,\cdot,t-t'))_{z} dz,
$$
the operator $L$ acting in $z$.
Since the singularities of $p$ and $Lp$ are relatively weak, we can solve the integral equation
by the standard method of successive approximation.
The kernel $\psi (x,y,t)$ obtained in this way is easily verified to be of class $C^\infty$, when
$t > 0$, and,
$$
   u (x,t)
 = \int_{\cX} (u_0, \ast^{-1} \psi (x,\cdot,t))_{y} dy
$$
satisfies
   $Lu = 0$ for $t > 0$
and
   $u (x,0) = u_0 (x)$ for all $x \in \cX$.

As for the estimates, note that if we define, relative to a Riemannian metric $ds$ on $\cX$, a
distance function to be
$$
   d (x,y) = \inf \int_{\widehat{xy}} ds
$$
for $x, y \in \cX$, where $\widehat{xy}$ is a path from $x$ to $y$, then
\begin{equation}
\label{eq.eotfsps}
   |\partial_x^{\alpha} \partial_y^{\beta} \psi (x,y,t)|
 \leq
   c\,
   \frac{1}{t^{\frac{\scriptstyle{n + |\alpha| + |\beta|}}{\scriptstyle{2m}}}}\,
   \exp \Big( - c' \frac{(d (x,y))^{(2m)'}}{t^{(2m)' - 1}}\Big)
\end{equation}
locally on $\cX'$ and for each entry of the matrix, where $(2m)'$ is the conjugate number
for $2m$ given by $1/(2m) + 1/(2m)' = 1$.
Also, as can be inferred and in fact easily proved from (\ref{eq.paramet}) and (\ref{eq.eotfsps}),
we get
$$
   |\psi (x,y,t) - p_U (x,y,t)|
 \leq
   c\,
   \frac{1}{t^{\frac{\scriptstyle{n-1}}{\scriptstyle{2m}}}}\,
   \exp \Big( - c' \frac{(d (x,y))^{(2m)'}}{t^{(2m)' - 1}}\Big)
$$
uniformly on each compact subset of $U \times U$, showing that locally $p_U$ yields a first
approximation to the fundamental solution $\psi (x,y,t)$.

\begin{lemma}
\label{l.Greenhe}
Suppose that $u$ is a section of $F^i$ over $\cC$ of Slobodetskij class
$H^{\bs (1)} (\cC)$.
Then
\begin{equation}
\label{eq.Greenhe}
   u (x,t)
 = (\varPsi u (\cdot,0)) (x,t)
 + \int_0^t dt' \int_{\cX} (Lu (\cdot,t'), \ast^{-1} \psi (x,\cdot,t-t'))_{y} dy
\end{equation}
for all $(x,t) \in \cC$.
\end{lemma}

\begin{proof}
Since $C^\infty (\overline{\cC}, F^i)$ is dense in $H^{\bs (1)} (\cC, F^i)$, it
suffices to establish the formula for those sections $u$ which are $C^\infty$ in
the closure of $\cC$.
As is well known, the Cauchy problem
\begin{equation}
\label{eq.Cpftheq}
\Big\{
\begin{array}{rclcl}
   Lu
 & =
 & f
 & \mbox{in}
 & \cC,
\\
   u (\cdot,0)
 & =
 & u_0
 & \mbox{on}
 & \cX
\end{array}
\end{equation}
has a unique solution in $C^\infty (\overline{\cC}, F^i)$ for all smooth data $f$
in $\overline{\cC}$ and $u_0$ on $\cX$.
By the above, the first term on the right-hand side of (\ref{eq.Greenhe}) is a solution
of the Cauchy problem with $f = 0$ and $u_0 = u (\cdot,0)$.
Hence, we shall have established the lemma if we prove that the second term is a solution
of the Cauchy problem with $f = Lu$ and $u_0 = 0$.
To this end, we rewrite the second summand on the right-hand side of (\ref{eq.Greenhe}) in
the form
$$
   \int_0^t \varPsi \left( Lu (\cdot,t') \right) (x,t-t')\, dt'
$$
for $(x,t) \in \cC$.
Obviously, the initial value at $t = 0$ of this integral vanishes, and so it remains to
calculate its image by $L$.
We get
\begin{eqnarray*}
\lefteqn{
   L \int_0^t \varPsi \left( Lu (\cdot,t') \right) (x,t-t')\, dt'
}
\\
 & = &
   \varPsi \left( Lu (\cdot,t) \right) (x,0)
 + \int_0^t L\, \varPsi \left( Lu (\cdot,t') \right) (x,t-t')\, dt'
\\
 & = &
   Lu (x,t),
\end{eqnarray*}
as desired.
\end{proof}

\section{Reduction to an integral equation}
\label{s.rtaiequ}

Given integrable sections $f$ and $u_0$ of $F^i$ over $\cC$ and $\cX$, respectively, we introduce
two potentials
$$
\begin{array}{rcl}
   \mathcal{P}_v (f)\, (x,t)
 & =
 & \displaystyle
   \int_0^t \varPsi \left( f (\cdot,t') \right) (x,t-t')\, dt',
\\
   \mathcal{P}_i (u_0)\, (x,t)
 & =
 & (\varPsi u_0)\, (x,t)
\end{array}
$$
for $(x,t) \in \cC$.
They are said to be parabolic volume and initial value potentials, respectively.

\begin{lemma}
\label{l.bounopp}
As defined above, the parabolic potentials induce continuous linear mappings
$$
\begin{array}{rrcl}
   \mathcal{P}_v\, :
 & L^2 (\cC, F^i)
 & \to
 & H^{\bs (1)} (\cC, F^i),
\\
   \mathcal{P}_i\, :
 & H^m (\cX, F^i)
 & \to
 & H^{\bs (1)} (\cC, F^i).
\end{array}
$$
\end{lemma}

\begin{proof}
Let $f \in L^2 (\cC,F^i)$.
Define $u = \mathcal{P}_v (f)$.
The proof of Lemma~\ref{l.Greenhe} shows that $u$ is a solution of problem (\ref{eq.Cpftheq})
with $u_0 = 0$.
By the maximal regularity theorem of \cite{Slob58}, we get $u \in H^{\bs (1)} (\cC, F^i)$.
Moreover, $u$ depends continuously of $f$, which proves the continuity of $\mathcal{P}_v$.
Analogously one establishes the continuity of the initial value potential.
\end{proof}

In terms of potentials formula (\ref{eq.Greenhe}) looks like
$
   u = \mathcal{P}_i (u (\cdot,0)) + \mathcal{P}_v (Lu)
$
in the cylinder $\cC$.
This leads immediately to a reformulation of the Cauchy problem of Lemma \ref{l.Stokes}.

\begin{theorem}
\label{t.reducti}
A section $u \in H^{\bs (1)} (\cC,F^i)$ is a solution of Cauchy problem (\ref{eq.Stokes})
if and only if it satisfies
\begin{equation}
\label{eq.reducti}
   u + \mathcal{P}_v (P N (u)) = \mathcal{P}_i (v_0) + \mathcal{P}_v (Pf)
\end{equation}
in $\cC$.
\end{theorem}

\begin{proof}
{\it Necessity\,}
Suppose that $u \in H^{\bs (1)} (\cC,F^i)$ is a solution of the Cauchy problem
$$
\begin{array}{rclll}
   u'_t + \varepsilon\, P \iD^m u + \nu\, P (\iD u) + P N (u)
 & =
 & Pf
 & \mbox{in}
 & \cC,
\\
   u (\cdot,0)
 & =
 & v_0
 & \mbox{on}
 & \cX
\end{array}
$$
in the cylinder.
Then $A^\ast u = 0$, which is due to Lemma~\ref{l.Stokes}, and so on applying
Lemma~\ref{l.Helmholtz2} we see that $Pu = u$.
Since the operators $\iD$ and $P$ commute, it follows that $Lu = Pf - P N (u)$ in $\cC$.
By (\ref{eq.Greenhe}),
$
   u = \mathcal{P}_i (v_0) + \mathcal{P}_v (Pf - P N (u))
$
in $\cC$, showing (\ref{eq.reducti}).

{\it Sufficiency\,}
Let $u \in H^{\bs (1)} (\cC,F^i)$ satisfy equation (\ref{eq.reducti}).
By the very definition of parabolic potentials, we get
$$
\begin{array}{rclll}
   Lu
 & =
 & Pf - P N (u)
 & \mbox{in}
 & \cC,
\\
   u (\cdot,0)
 & =
 & v_0
 & \mbox{on}
 & \cX.
\end{array}
$$
Since $v_0$ satisfies $A^\ast v_0 = 0$ on $\cX$, we conclude that $P v_0 = v_0$.
Therefore, applying the projection $P$ to both sides of the above equalities yields
   $L (u - Pu) = 0$ in $\cC$ and
   $(u - Pu) (\cdot,0) = 0$ on $\cX$.
By uniqueness, $Pu = u$ in the cylinder, and so $u$ is a solution of Cauchy problem
(\ref{eq.Stokes}).
\end{proof}

Note that (\ref{eq.reducti}) is a nonlinear integral equation of Volterra type for
the unknown section $u$ in $H^{\bs (1)} (\cC, F^i)$.
The nonlinear operator $N$ maps $H^{\bs (1)} (\cC,F^i)$ continuously and compactly into
$L^2 (\cC,F^i)$ while $P$ is a continuous mapping of $L^2 (\cC,F^i)$ into itself and
$\mathcal{P}_v$ maps $L^2 (\cC,F^i)$ continuously into $H^{\bs (1)} (\cX,F^i)$, the
latter being due to Lemma \ref{l.bounopp}.
Hence it follows that the composition $K = \mathcal{P}_v \circ P \circ N$ is a compact
continuous operator in $H^{\bs (1)} (\cC,F^i)$.
Thus, equation (\ref{eq.reducti}) is of the Leray-Schauder type $(\Id + K) u = f$, where
moreover $K$ is differentiable.
Our next concern will be to show that the mapping $\Id + K$ is actually a diffeomorphism
of $H^{\bs (1)} (\cC,F^i)$ for each $\varepsilon > 0$.

\section{Comments on the Navier-Stokes equations for quasicomplexes}
\label{eq.coNSefq}

For any complex (\ref{eq.ellcomp}), the composition $A^{i} A^{i-1}$ of any two neighbouring
operators vanishes.
Using the language of differential geometry one can think of the composition $A \circ A$ as
the curvature of (\ref{eq.ellcomp}).
In this way the complexes are flat sequences of topological vector spaces and their continuous
mappings.
Differential geometry gives a convincing motivation to those sequences which fail to be flat
but whose curvature is small in some sense.
Let alone the extended sequence of connections associated to any smooth vector bundle, whose
curvature proves to be a bundle morphism.
More generally, we call (\ref{eq.ellcomp}) a quasicomplex if each composition $A^{i} A^{i-1}$
is actually a pseudodifferential operator of order one.
Given any quasicomplex, the sequence of principal symbols still constitutes a complex, and so
the ellipticity survives under ``small'' perturbations of complexes.
The Laplacians
$
   \iD^i = A^i{}^\ast A^i + A^{i-1} A^{i-1}{}^\ast
$
of an elliptic quasicomplex are nonnegative symmetric second order elliptic operators between
sections of vector bundles $F^i$, for $i = 0, 1, \ldots, N$.
Therefore, we are in a position to introduce the generalised Navier-Stokes equations also for
elliptic quasicomplexes in much the same way as in (\ref{eq.mhdefivf}).
Note, however, that neither $A$ nor $A^\ast$ commute with $\iD$, and so the elimination of the
pressure is no longer possible using the Neumann problem after Spencer, as it is done in
   Section \ref{s.eliotpr}.

\part{Solvability of regularised Navier-Stokes equations}
\label{p.sorNSeq}

We now return to problem (\ref{eq.gNSefivf}).
For this problem it is possible to prove global existence and uniqueness of the weak as
well as strong solutions.
The characterisation of weak solutions is analogous to that of \cite{Hopf51}.
The description of strong solutions is similar to that of \cite{KiseLady57}.
The characterisation of strong solutions to the conventional system in  \cite{KiseLady57}
can be improved using the eigenfunctions of the Stokes operator
$$
   \Big( \begin{array}{cc}
           \nu \iD
         & A
\\
           A^\ast
         & 0
         \end{array}
   \Big)
$$
and taking the estimates of \cite{Prod62} but we will not develop this point here.
A regularisation similar to (\ref{eq.gNSefivf}) was first studied by Ladyzhenskaya in
\cite{Lady62}.
She established a convergence theorem for the strong solutions of this system to those
of the Navier-Stokes equation as the regularisation parameter tends to zero.
Lions proposed a regularisation of the Navier-Stokes equations by adding a higher order
viscosity term $\varepsilon (- \iD)^m u$ and requiring the solution to vanish up to
order $m$ on the lateral boundary of the cylinder, see Remark 6.11 in \cite{Lion69}.
For this problem he proved the global existence of unique weak solutions.
The paper \cite{Veig85a} concerns regularisations of this type in the whole space $\R^n$
and establishes convergence theorems as the regularisation parameter approaches zero,
see also \cite{Veig85b}.

\section{Governing equations and functional framework}
\label{s.geaffra}

To proceed we first make more precise our assumptions on the structure of nonlinearity in
problem (\ref{eq.gNSefivf}).
Namely, we consider $N (u,v) := T (v) u$ for differentiable sections $u$ und $v$ of $F^i$,
where $T$ is a first order differential operator on $\cX$ mapping sections of $F^i$ into
sections of the bundle $\mathrm{Hom}\, F^i$ of homomorphisms of $F^i$.
Our standing assumption on the differential operators $T$ under study is that there is a
sesquilinear mapping $Q (v,w)$ of $C^\infty (\cX, F^i) \times C^\infty (\cX, F^i)$ into
$C^\infty (\cX,F^{i-1})$, such that
   $(T (v))^\ast v = A Q (v,v)$
for all smooth sections $v$ of $F^i$.
For the conventional nonlinearity, we get $T (v) = v'_x$ and
$$
   Q (v,w) = \frac{1}{2}\, (v,w)_x,
$$
as is easy to check.
As usual, we abbreviate $N (u,u)$ to $N (u)$, and similarly for $Q (v,v)$.

\begin{lemma}
\label{l.algestr}
As defined above, the differential operator $T$ satisfies the anticommutativity relation
\begin{equation}
\label{eq.algestr}
   (T (v))^\ast w + (T (w))^\ast v
 = A \left( Q (v) - Q (v-w) + Q (w) \right)
\end{equation}
for all $v, w \in C^\infty (\cX,F^i)$.
\end{lemma}

\begin{proof}
Since $T$ is linear, it follows that
\begin{eqnarray*}
   (T (v))^\ast w + (T (w))^\ast v
 & = &
   (T (v))^\ast v - (T (v-w))^\ast (v-w) + (T (w))^\ast w
\\
 & = &
   A \left( Q (v) - Q (v-w) + Q (w) \right),
\end{eqnarray*}
as desired.
\end{proof}

The following lemma clarifies why (\ref{eq.algestr}) is referred to as an anticommutativity
relation.

\begin{corollary}
\label{c.algestr}
Assume that $u \in C^\infty (\cX, F^i)$ satisfies $A^\ast u = 0$ and has zero Cauchy data on
the boundary of $\cX$ with respect to the differential operator $A^\ast$.
Then, one has
$$
   (N (u,v), w) = - (N (u,w), v)
$$
for all $v, w \in C^\infty (\cX,F^i)$.
\end{corollary}

\begin{proof}
Indeed, we get
\begin{eqnarray*}
   (N (u,v), w) + (N (u,w), v)
 & = &
   (T (v)\, u, w) + (T (w)\, u, v)
\\
 & = &
   (u, (T (v))^\ast w) + (u, (T (w))^\ast v)
\\
 & = &
   (u, A \left( Q (v) - Q (v-w) + Q (w) \right))
\\
 & = &
   0,
\end{eqnarray*}
as desired.
\end{proof}

The equality $(N (u,v), w) = - (N (u,w), v)$ is actually a crucial property characterising
those nonlinear terms $N (u,v)$ which are permitted in this work.
In particular, we see that $(N (u), u) = 0$ holds for any smooth solenoidal section $u$ of
the bundle $F^i$.

\begin{example}
\label{e.algestr}
For the conventional Navier-Stokes equations on a compact Riemannian manifold with boundary
$\cX$ one chooses $N (u,v) = \partial_u v$, see \cite{EbinMars70}.
Here, we interpret $u$ and $v$ as one-forms on $\cX$ and $\partial$ stands for a connection
on (the tangent bundle of) the manifold, so that $\partial_u v$ is the derivative of $v$ in
the direction $u$.
By definition we get $\partial_u v = \iota (u) \partial v$, where by $\iota (u)$ is meant
the interior multiplication by $u$.
Show that if $\partial$ is compatible with the Riemannian metric of $\cX$, then the equality
$(N (u,v), w) = - (N (u,w), v)$ is fulfilled.
Indeed, let $u$ be a differentiable one-form on $\cX$ satisfying $d^\ast u = 0$ in $\cX$ and
whose normal part on the boundary vanishes.
Then
\begin{eqnarray*}
   (N (u,v), w)
 & = &
   \int_{\cX} (N (u,v),w)_x dx
\\
 & = &
   \int_{\cX} (\iota (u) \partial v, w)_x dx
\\
 & = &
   \int_{\cX} \iota (u)\, (\partial v, w)_x dx
\\
 & = &
   \int_{\cX} \iota (u) \big( (\partial v, w)_x + (v, \partial w)_x \big) dx
 - \int_{\cX} \iota (u)\, (v, \partial w)_x dx
\\
 & = &
   \int_{\cX} \iota (u)\, d (v, w)_x dx
 - \int_{\cX} \iota (u)\, (\partial w, v)_x dx
\\
 & = &
   \int_{\cX} d^\ast u\, (v, w)_x dx
 - (N (u,w),v)
\\
 & = &
   -\, (N (u,w), v)
\end{eqnarray*}
for all $v, w \in \varOmega^1 (\cX)$, as desired.
In particular, this is the case for the Levi-Civita connection on $\cX$.
\end{example}

Denote by $\mathcal{S}$ the space of all sections $u \in C^\infty (\cX,F^i)$ which vanish
up to the infinite order on the boundary of $\cX$ and satisfy $A^{\ast} u = 0$ in $\cX$.
Note that while studying the problem for compact closed manifolds we present the framework
to suit to compact manifolds with boundary as well.
Let $H$ stand for the closure of $\mathcal{S}$ in $L^2 (\cX, F^i)$.
When equipped with the induced unitary structure, $H$ is a Hilbert space.
The elements of $H$ are still solenoidal and inherit the property $n (u) = 0$ on the boundary
of $\cX$ in a weak sense, $n (u)$ being the Cauchy data of $u$ with respect to $A^\ast$.
Furthermore, we write $S = S^m$ for the closure of $\mathcal{S}$ in $H^{m} (\cX,F^i)$.
Using spectral synthesis in Sobolev spaces one sees that $S$ coincides with the subspace of
$H^{m} (\cX,F^i)$ consisting of those solenoidal sections $u$ which vanish up to order $m$
on the boundary of $\cX$.

\begin{lemma}
\label{l.Dirichl}
The norm in $S$ inherited from $H^m (\cX,F^i)$ is equivalent to that determined by the
energy (or Dirichlet) inner product
\begin{equation}
\label{eq.Dirichl}
   D (u,v)
 = \varepsilon\, (\iD^{m/2} u, \iD^{m/2} v)
 + \nu \left( (Au, Av) + (A^\ast u, A^\ast v) \right)
 + (u,v)
\end{equation}
for $u, v \in S$.
\end{lemma}

\begin{proof}
This follows immediately from the fact that the Dirichlet problem for the differential
operator
   $V = \varepsilon \iD^m + \nu \iD$
is elliptic in the classical setting of Sobolev spaces.
\end{proof}

The space $S$ is contained in $H$, it is dense in $H$, and the embedding is continuous.
Let $H'$ and $S'$ denote the dual spaces of $H$ and $S$, respectively, and let $\iota$
denote the embedding of $S$ into $H$.
The transposed operator $\iota'$ maps $H'$ continuously into $S'$, and
  it is one-to-one since $\iota (S) = S$ is dense in $H$, and
  $\iota' (H')$ is dense in $S'$ since $\iota$ is one-to-one.
Moreover, by the Riesz representation theorem, we can identify $H$ and $H'$, and so we arrive
at the continuous inclusions
$
   S \hookrightarrow H \cong H' \hookrightarrow S',
$
each space being dense in the subsequent one.
As a consequence of these identifications, the inner product in $H$ can be specified within
the pairing of $S'$ and $S$.
More precisely, there is a sesquilinear pairing $S' \times S \to \mathbb{C}$ (for which we
continue to write $(f,v)$), such that
$$
   (f, \iota (v)) = \overline{\langle \iota' \ast_H f, v \rangle}
$$
for all
   $f \in H$ and
   $v \in S$,
where
   $\ast_H : H \to H'$
is the conjugate linear isomorphism given by the Riesz theorem.

For each $u \in S$, the form $v \to D (u,v)$ is conjugate linear and continuous on $S$.
Therefore, there is an element of $S'$ (which we denote by $(V + \Id) u$) with the property
that
\begin{equation}
\label{eq.Frieext}
   ((V + \Id) u,v) = D (u,v)
\end{equation}
for all $v \in S$.
In this way we get what is known as the Friedrichs extension of the differential operator
   $\varepsilon \iD^m + \nu \iD$,
cf. \cite{AlfoSimo80}.
The mapping $u \mapsto (V + \Id) u$ is linear and continuous and so is $u \mapsto Vu$.
Moreover, $V + \Id$ is an isomorphism from $S$ onto $S'$.

Let $\mathcal{B}$ be any Banach space.
For a finite real number $p \geq 1$, we denote by $L^{p} (I,\mathcal{B})$ the space of all
measurable functions $u$ on the interval $I$ with values in $\mathcal{B}$, such that
\begin{equation}
\label{eq.Bochner}
   \Big( \int_I \| u (t) \|_{\mathcal{B}}^p\, dt \Big)^{1/p} < \infty.
\end{equation}
This space is Banach under the norm (\ref{eq.Bochner}).
The space $L^\infty (I,\mathcal{B})$ is defined to consist of all essentially bounded
functions $u$ of $t \in I$ with values in $\mathcal{B}$.
As usual, this space is given the norm obtained from (\ref{eq.Bochner}) by passing to the
limit as $p \to \infty$.
It is a Banach space.

By $C (\overline{I},\mathcal{B})$ is meant the space of all continuous functions on $I$ with
values in $\mathcal{B}$.
If $I$ is bounded, the space is topologised under norm induced by the embedding into
$L^\infty (I,\mathcal{B})$, i.e.,
$$
   \| u \|_{C (\overline{I},\mathcal{B})}
 = \sup_{t \in I} \| u (t) \|_{\mathcal{B}}.
$$

The following technical lemma concerns the derivatives of functions with values in Banach
spaces.

\begin{lemma}
\label{l.primiti}
Assume $U$ and $u$ are two functions of $L^1 (I,\mathcal{B})$.
Then, the following are equivalent:

1)
$U$ is a.e. equal to a primitive function of $u$, i.e.,
$
   \displaystyle
   U (t) = U_0 + \int_0^t u (t') dt'
$
for a.a. $t \in I$.

2)
$U' = u$ weakly in $I$, i.e.,
$
   \displaystyle
   \int_I U (t) \phi' (t) dt = - \int_I u (t) \phi (t) dt
$
is valid for each $\phi \in C^\infty_{\mathrm{comp}} (0,T)$.

3)
For each $f \in \mathcal{B}'$, the equality
$
   \displaystyle
   \frac{d}{dt} \langle f, U \rangle = \langle f, u \rangle
$
holds weakly in $I$.
\end{lemma}

If one of the conditions 1)-3) is satisfied, then $u$ is, in particular, a.e. equal to
a continuous function on $I$ with values in $\mathcal{B}$.

\begin{proof}
See for instance \cite[p.~169]{Tema79}.
\end{proof}

\section{The basic linear problem}
\label{s.basiclp}

The basic linear problem relating to the generalised Navier-Stokes equations of (\ref{eq.mhdefivf})
consists in finding sections
   $u$ and
   $p$
of $F^i$ and $F^{i-1}$ over $\cC$, respectively, such that
\begin{equation}
\label{eq.basiclp}
\begin{array}{rcl}
   \partial_t u + (\varepsilon \iD^m + \nu \iD) u + Ap
 & =
 & f,
\\
   A^\ast u
 & =
 & 0
 \end{array}
\end{equation}
in the cylinder $\cC$ and
\begin{equation}
\label{eq.mixedbc}
\begin{array}{rclcl}
   u (\cdot,0)
 & =
 & u_0
 & \mbox{on}
 & \cX,
\\
   u
 & =
 & 0
 & \mbox{up to order $m$ on}
 & \partial \cX \times I,
 \end{array}
\end{equation}
where
   $f$ is a section of $F^i$ over the cylinder and
   $u_0$ a section of $F^i$ over its bottom $\cX$.
Conditions (\ref{eq.mixedbc}) specify the first mixed problem for evolution equations
(\ref{eq.basiclp}).
In the case of conventional Navier-Stokes equations this problem is referred to as the
Stokes problem.

Let $(u,p)$ be a classical solution of problem (\ref{eq.basiclp}), (\ref{eq.mixedbc}),
say
   $u \in C^{2m} (\overline{\cC},F^i)$ and
   $p \in C^1 (\overline{\cC},F^{i-1})$.
Then
$$
   (u'_t, v) + (Vu, v) = (f, v)
$$
for each element $v$ of $\mathcal{S}$, as is easy to see.
By continuity, this equation holds also for each $v \in S$.
Since
   $(u'_t, v) = \partial_t\, (u,v)$,
we are led to the following weak formulation of problem (\ref{eq.basiclp}), (\ref{eq.mixedbc}).
Given sections
   $f \in L^2 (I,S')$ and
   $u_0 \in H$,
find $u \in L^2 (I,S)$ satisfying
\begin{equation}
\label{eq.Stokesw}
\begin{array}{rclcl}
   \partial_t\, (u,v) + (Vu,v)
 & =
 & (f,v)
 & \mbox{for all}
 & v \in S,
\\
   u (0)
 & =
 & u_0.
 &
 &
 \end{array}
\end{equation}
(By (\ref{eq.Stokes}), any solution to this problem leads to a solution in some weak
 sense of problem (\ref{eq.basiclp}), (\ref{eq.mixedbc}).)

For $u$ belonging to $L^2 (I,S)$, the condition $u (0) = u_0$ need not make any sense in
general.
However, if $u \in L^2 (I,S)$ satisfies the first equation of (\ref{eq.Stokesw}), then
the value $u (0)$ is well defined.
Indeed, we can write the first equality of (\ref{eq.Stokesw}) as
$
   \partial_t\, (u,v) = (f - Vu,v)
$
for all $v \in S$.
Since $V$ is linear and continuous from $S$ into $S'$ and $u \in L^2 (I,S)$, the function
$Vu$ belongs to $L^2 (I,S')$.
Hence it follows that $f - Vu \in L^2 (I,S')$, and Lemma \ref{l.primiti} shows that
   $u' \in L^2 (I,V')$
and that $u$ is a.e. equal to an (absolutely) continuous function on $I$ with values in
$S'$.
Thus, any function of $L^2 (I,S)$ satisfying the first equation of (\ref{eq.Stokesw}) is,
after modification on a set of measure zero, a continuous mapping of $I$ into $S'$, and
so the condition $u (0) = u_0$ makes sense.

If $f \in L^2 (I,S')$ and $u \in L^2 (I,S)$ satisfies the first equation of (\ref{eq.Stokesw})
then, as observed, $u' \in L^2 (I,S')$ and
$
   \partial_t\, (u,v) = (f - Vu,v)
$
for all $v \in S$.
According to Lemma \ref{l.primiti} this equality is equivalent to
$
   u' + Vu = f
$
a.e. on $I$.
Conversely, if $u$ satisfies the latter equation, then $u$ obviously satisfies the first
equation of (\ref{eq.Stokesw}).
In this way we arrive at an alternative formulation of weak problem
   (\ref{eq.Stokesw}).
Namely, given sections
   $f \in L^2 (I,S')$ and
   $u_0 \in H$,
find
   $u \in L^2 (I,S)$, such that
   $u' \in L^2 (I,S')$
and
\begin{equation}
\label{eq.Stokesp}
\begin{array}{rclcl}
   u' + Vu
 & =
 & f
 & \mbox{a.e. on}
 & I,
\\
   u (0)
 & =
 & u_0.
 &
 &
 \end{array}
\end{equation}

\begin{theorem}
\label{t.exiuniS}
For any data
   $f \in L^2 (I,S')$ and
   $u_0 \in H$,
there exists a unique function
   $u \in L^2 (I,S)$
which satisfies
   $u' \in L^2 (I,S')$
and (\ref{eq.Stokesp}).
Moreover, we get $u \in C (\overline{I},H)$.
\end{theorem}

\begin{proof}[Proof of the existence]
We use the Faedo-Galerkin method.
Since $S$ is separable, there is a sequence of linearly independent elements
   $( e_i )_{i = 1, 2, \ldots}$
which is complete in $S$.
For every $k = 1, 2, \ldots$ we define an approximate solution of problem
   (\ref{eq.Stokesw}) or
   (\ref{eq.Stokesp})
by
\begin{equation}
\label{eq.appsolS}
   u_k = \sum_{i=1}^k c_{k,i} (t) e_i
\end{equation}
and
\begin{equation}
\label{eq.appsolc}
\begin{array}{rclcl}
   (u_k',e_j) + (V u_k,e_j)
 & =
 & (f,e_j)
 & \mbox{for all}
 & j = 1, \ldots, k,
\\
   u_k (0)
 & =
 & u_{0,k},
 &
 &
 \end{array}
\end{equation}
where $u_{0,k}$ is, for example, the orthogonal projection in $H$ of $u_0$ on
the space spanned by the elements $e_1, \ldots, e_k$.
(Note that $u_{0,k}$ can be any element of the space spanned by $e_1, \ldots, e_k$,
 such that $u_{0,k} \to u_0$ in the norm of $H$, as $k \to \infty$.)

The coefficients $c_{k,i}$, $1 \leq i \leq k$, are scalar-valued functions on the
interval $I$, and (\ref{eq.appsolc}) is a system of linear ordinary differential
equations for these functions.
Indeed, we obtain
$$
   \sum_{i=1}^k (e_i,e_j)\, c_{k,i}' (t)
 + \sum_{i=1}^k (V e_i,e_j)\, c_{k,i} (t)
 = (f,e_j)
$$
for all $j = 1, \ldots, k$.
Since the elements $e_1, \ldots, e_k$ are linearly independent, the matrix with
entries
   $(e_i,e_j)$,
where $1 \leq i, j \leq k$, is regular.
Hence, on inverting this matrix we reduce the above system to a system of linear
ordinary equations with constant coefficients
\begin{equation}
\label{eq.resoode}
   c_{k,i}' (t)
 + \sum_{j=1}^k a_{i,j}\, c_{k,j} (t)
 = \sum_{j=1}^k b_{i,j}\, (f,e_j)
\end{equation}
for $i = 1, \ldots, k$, where $a_{i,j}$ and $b_{i,j}$ are complex numbers.
Furthermore, the condition $u_k (0) = u_{0,k}$ is equivalent to $k$ equations
   $c_{k,i} (0) = \pi_i (u_{0,k})$,
where $\pi_i$ is the projection of $u_{0,k}$ on the one-dimensional space spanned by
$e_i$.
System (\ref{eq.resoode}) together with the initial conditions determines uniquely the
coefficients $c_{k,i}$ on the whole interval $I$.

Since the scalar functions $t \mapsto (f (t),e_j)$ are square integrable, so are the
functions $c_{k,i}$.
Hence it follows that
   $u_k \in L^2 (I,S)$ and
   $u_k' \in L^2 (I,S)$
for each $k$.
We will obtain a priori estimates independent of $k$ for the functions $u_k$ and then
pass to the limit.

\textit{A priori estimates}.
To this end, we multiply the first equation of (\ref{eq.appsolc}) by $c_{k,j} (t)$ and
sum up these equations for $j = 1, \ldots, k$.
We get
$$
   (u_k' (t),u_k (t)) + (V u_k (t),u_k (t))
 = (f (t),u_k (t)).
$$
Since
$
   \partial_t (u_k (t),u_k (t)) = 2\, \Re\, (u_k' (t),u_k (t)),
$
it follows that
$$
   \partial_t\, \| u_k (t) \|_H^2 + 2 \left( D (u_k (t)) - \| u_k (t) \|_H^2 \right)
 = 2\, \Re\, (f (t),u_k (t)),
$$
where $D (u) := D (u,u)$ and $\sqrt{D (u)}$ is the equivalent norm in $S$.
The right-hand side of this equality is majorised by
$$
   2\, \| f (t) \|_{S'} \| u_k (t) \|_S
 \leq
   c\, \| f (t) \|_{S'}^2 + \frac{1}{c}\, D (u_k (t))
$$
with arbitrary constant $c > 0$.
We choose $c$ in such a way that $2 - 1/c > 0$.
On the other hand,
$$
   \partial_t\, \| u_k (t) \|_H^2 - 2 \| u_k (t) \|_H^2
 = e^{2t}\, \partial_t \left( e^{- 2t} \| u_k (t) \|_H^2 \right),
$$
as is easy to check.
Therefore,
\begin{equation}
\label{eq.tlhsoti}
   \partial_t \left( e^{- 2t} \| u_k (t) \|_H^2 \right)
 + \Big( 2 - \frac{1}{c} \Big) e^{- 2t}\, D (u_k (t))
 \leq
   c\, e^{- 2t}\, \| f (t) \|_{S'}^2.
\end{equation}
Integrating the inequality from $0$ to $t$, where $t$ is a fixed point of $I$, we readily
deduce that
\begin{eqnarray*}
   \| u_k (t) \|_H^2
 & \leq &
   e^{2t}\, \| u_{0,k} \|_H^2
 + c \int_0^t e^{2 (t-t')} \| f (t') \|_{S'}^2 dt'
\\
 & \leq &
   e^{2t}\, \| u_0 \|_H^2
 + c \int_0^t e^{2 (t-t')} \| f (t') \|_{S'}^2 dt'
\end{eqnarray*}
for almost all $t \in I$.
Therefore,
$$
   \sup_{t \in I} \| u_k (t) \|_H^2
 \leq
   e^{2T}
   \Big( \| u_0 \|_H^2 + c \int_0^T \| f (t) \|_{S'}^2 dt \Big).
$$
the right-hand side being finite and independent of $k$.
We have thus proved that the sequence $u_k$ remains in a bounded subset of $L^\infty (I,H)$,
i.e., there is a constant $C$ such that
\begin{equation}
\label{eq.boinLiH}
   \| u_k \|_{L^\infty (I,H)} \leq C
\end{equation}
for all $k$.

Furthermore, on integrating inequality (\ref{eq.tlhsoti}) in $t$ over all of $I$ we arrive at
the estimates
\begin{eqnarray*}
   \| u_k (T) \|_H^2
 + \Big( 2 - \frac{1}{c} \Big)
   \int_0^T \!\!\! e^{2 (T-t)} D (u_k (t)) dt
 \!\! & \!\! \leq \!\! & \!\!
   e^{2T} \| u_{0,k} \|_H^2
 + c \int_0^T \!\!\! e^{2 (T-t)} \| f (t) \|_{S'}^2 dt
\\
 \!\! & \!\! \leq \!\! & \!\!
   e^{2T}\, \| u_0 \|_H^2
 + c \int_0^T \!\!\! e^{2 (T-t)} \| f (t) \|_{S'}^2 dt
\end{eqnarray*}
for each $k = 1, 2, \ldots$.
This shows that the sequence $u_k$ remains in a bounded subset of $L^2 (I,S)$, i.e.,
\begin{equation}
\label{eq.boinLiS}
   \| u_k \|_{L^2 (I,S)} \leq C
\end{equation}
for all $k$, with $C$ a constant independent of $k$.

\textit{Passage to the limit}.
The a priori estimate of (\ref{eq.boinLiH}) implies that there is a subsequence of
$u_k$ which converges in the weak$^\ast$ topology of $L^\infty (I,H)$.
By abuse of notation, we continue to write $u_k$ for this subsequence.
Thus, there is an element $u \in L^\infty (I,H)$ such that
$$
   \int_I (u_k (t) - u (t), v (t))\, dt \to 0
$$
for all $v \in L^1 (I,H)$.
By (\ref{eq.boinLiH}), the subsequence $u_k$ belongs to a bounded set in $L^2 (I,S)$.
Therefore, another passage to a subsequence shows the existence of some $\tilde{u}$
in $L^2 (I,S)$ and a subsequence (still denoted by $u_k$) which converges to $\tilde{u}$
in the weak$^\ast$ topology of $L^2 (I,S)$.
(Note that the weak$^\ast$ and weak topologies of $L^2 (I,S)$ coincide.)
This means
$$
   \int_I (u_k (t) - \tilde{u} (t), v (t))\, dt \to 0
$$
for all $v \in L^2 (I,S')$.
This latter relation is fulfilled, in particular, if $v \in L^2 (I,H)$.
Since $L^2 (I,H) \hookrightarrow L^1 (I,H)$, it follows that
$$
   \int_I (u (t) - \tilde{u} (t), v (t))\, dt = 0
$$
whenever $v \in L^2 (I,H)$.
Hence, $u = \tilde{u}$ belongs to $L^\infty (I,H) \cap L^2 (I,S)$.

In order to pass to the limit in equations (\ref{eq.appsolc}), we will consider
scalar functions $\phi$ of $t \in I$, which are continuously differentiable and
satisfy $\phi (T) = 0$.
Given such a function $\phi$, we multiply the first equality of (\ref{eq.appsolc})
by $\phi (t)$, integrate in $t \in I$ and use the formula
$$
   \int_0^T (u_k' (t), e_j)\, \phi (t) dt
 = - \int_0^T (u_k (t), e_j)\, \phi' (t) dt - (u_k (0), e_j)\, \phi (0).
$$
In this way we find
$$
 - \int_0^T \!\! (u_k (t), e_j) \phi' (t) dt
 + \int_0^T \!\! (V u_k (t), e_j) \phi (t) dt
 = (u_{0,k}, e_j)\, \phi (0)
 + \int_0^T \!\! (f (t), e_j) \phi (t) dt
$$
for all $j = 1, \ldots, k$.

The passage to the limit for $k \to \infty$ in the integrals on the left-hand side
is easy, for the sequence $u_k$ converges to $u$ in the weak${}^\ast$ topology of
both
   $L^\infty (I,H)$ and
   $L^2 (I,S)$.
We recall that $u_{0,k}$ converges to $u_0$ strongly in $H$.
Therefore, we find in the limit
$$
 - \int_0^T \!\! (u (t), e_j) \phi' (t) dt
 + \int_0^T \!\! (V u (t), e_j) \phi (t) dt
 = (u_0, e_j)\, \phi (0)
 + \int_0^T \!\! (f (t), e_j) \phi (t) dt.
$$
This equality, which holds for each $j$, allows one to write by a linearity argument
\begin{equation}
\label{eq.tewhfej}
 - \int_0^T \!\! (u (t), v)\, \phi' (t) dt
 + \int_0^T \!\! (V u (t), v)\, \phi (t) dt
 = (u_0, v)\, \phi (0)
 + \int_0^T \!\! (f (t), v)\, \phi (t) dt
\end{equation}
for all $v$ which are finite linear combinations of functions $e_j$.
Since every term of (\ref{eq.tewhfej}) depends linearly and continuously in the norm
of $V$ on $v$, the equality (\ref{eq.tewhfej}) is still valid, by continuity, for all
$v$ in $V$.

On writing (\ref{eq.tewhfej}), in particular,
   with any $\phi \in C^\infty_{\mathrm{comp}} (0,T)$,
we arrive at the equality
$$
   \partial_t (u,v) + (Vu,v) = (f,v)
$$
for all $v \in V$.
This is precisely (\ref{eq.Stokesw}) which is valid in the sense of distributions on
$(0,T)$.
As is shown immediately after (\ref{eq.Stokesw}), this equality together with
   $u \in L^2 (I,S)$
implies that
   $u'$ belongs to $L^2 (I,S')$ and
   $u' + Vu = f$
a.e. on $I$.

Finally, it remains to check that $u (0) = u_0$.
(The continuity of $u$ will be proved after the proof of existence.)
For this purpose we multiply the first equation of (\ref{eq.Stokesw}) by $\phi (t)$
with the same $\phi (t)$ as before, integrate in $t \in I$ and use the integration
by parts formula
$$
   \int_0^T \partial_t\, (u (t), v)\, \phi (t) dt
 = - \int_0^T (u (t), v)\, \phi' (t) dt - (u (0), v)\, \phi (0).
$$
We get
$$
 - \int_0^T \!\! (u (t), v)\, \phi' (t) dt
 + \int_0^T \!\! (V u (t), v) \phi (t) dt
 = (u (0), v)\, \phi (0)
 + \int_0^T \!\! (f (t), v)\, \phi (t) dt,
$$
and so comparing this with (\ref{eq.tewhfej}) yields
$
   (u (0) - u_0, v)\, \phi (0) = 0
$
for all $v \in S$ and for each function $\phi$ of the type considered.
Choose $\phi$ such that $\phi (0) \neq 0$, then
$
   (u (0) - u_0, v) = 0
$
for all $v \in S$.
This equality implies $u (0) = u_0$ and completes the proof of the existence.
\end{proof}

\begin{proof}[Proof of the continuity and uniqueness]
The proof is based on a lemma which is a particular case of a general interpolation
theorem of \cite{LionMage72}.

\begin{lemma}
\label{l.interLM}
Let $S$, $H$ and $S'$ be three Hilbert spaces, each space being densely included
into the subsequent one, and $S'$ being the dual of $S$.
If $u$ is a function in $L^2 (I,S)$ and its derivative $u'$ belongs to $L^2 (I,S')$,
then $u$ is equal a.e. to a continuous function on $I$ with values in $H$.
Moreover, $\partial (u,u)_H = 2 \Re\, (u',u)_H$ holds in the sense of distributions
on $(0,T)$.
\end{lemma}

Note that the equality $\partial (u,u)_H = 2 \Re\, (u',u)_H$ is meaningful since
the functions
   $t \mapsto \| u (t) \|_H^2$ and
   $t \mapsto (u' (t),u (t))_H$
are both integrable on $I$.

\begin{proof}
For a more elementary proof than that of \cite{LionMage72} we refer the reader to
\cite[p.~177]{Tema79}.
\end{proof}

By Lemma \ref{l.interLM}, the continuity of the mapping
   $u : I \to H$
in Theorem \ref{t.exiuniS} becomes obvious.
It only remains to check the uniqueness.
Suppose $u_1$ and $u_2$ are two solutions of problem (\ref{eq.Stokesp}) which bear
the desired properties.
Set $u = u_1 - u_2$, then
$$
\begin{array}{rcl}
    u & \in & L^2 (I,S),
\\
   u' & \in & L^2 (I,S')
\end{array}
$$
and
$$
\begin{array}{rclcl}
   u' + Vu
 & =
 & 0
 & \mbox{a.e. on}
 & I,
\\
     u (0)
 & =
 & 0.
 &
 &
\end{array}
$$
Taking the scalar product of the first equality with $u (t)$ yields
$$
   (u' (t), u (t)) + (Vu (t),u(t)) = 0
$$
a.e. on $I$.
Using Lemma \ref{l.interLM} we get
\begin{eqnarray*}
   \frac{1}{2}\, \partial_t (u (t),u (t))
 & =
 & -\, (Vu (t),u (t))
\\
 & \leq
 & 0
\end{eqnarray*}
whence
$
   \| u (t) \|_H^2 \leq \| u (0) \|_H^2 = 0
$
for all $t \in I$.
Hence, $u_1 (t) = u_2 (t)$ for each $t \in I$, as desired.
\end{proof}

On using the solution $u$ of problem (\ref{eq.Stokesp}) given by Theorem \ref{t.exiuniS}
we determine the pressure $p$ from (\ref{eq.sefp}).
This leads to a solution $(u,p)$ of linearised problem (\ref{eq.basiclp}).
Assuming that both $f$ and $u_0$ are sufficiently smooth, we can actually obtain as
much regularity as desired for $u$ and $p$.
We establish only a simple result of this type.

\begin{theorem}
\label{t.regtflNS}
Suppose
   $f \in L^2 (I,H)$ and
   $u_0 \in S$.
Then
   $u \in L^2 (I,H^{2m} (\cX,F^i))$ and
   $u' \in L^2 (I,H)$,
i.e.,
\begin{equation}
\label{eq.domStoc}
\begin{array}{rcl}
   u & \in & H^{\bs (1)} (\cC,F^{i}),
\\
   p & \in & L^2 (I,H^1 (\cX,F^{i-1})).
\end{array}
\end{equation}
\end{theorem}

Here, we tacitly assume that the Neumann problem after Spencer for complex (\ref{eq.ellcomp})
at step $i$ satisfies the Shapiro-Lopatinskij condition on the boundary of $\cX$, if there
is any.

\begin{proof}
The first point is to establish that $u' \in L^2 (I,H)$, i.e., $u' \in L^2 (\cC,F^i)$.
This is proved by deriving another a priori estimate for the approximate solution $u_k$
constructed by the Galerkin method.

Using the notation of the proof of Theorem \ref{t.exiuniS}, we multiply the first
equality of (\ref{eq.appsolc}) by the derivative $c_{k,j}' (t)$ and sum up these
equalities for $j = 1, \ldots, k$.
This gives
$$
   (u_k' (t),u_k' (t)) + (V u_k (t),u_k' (t))
 = (f (t),u_k' (t)).
$$
Since
$
   \partial_t (V u_k (t),u_k (t)) = 2\, \Re\, (V u_k (t),u_k' (t)),
$
we get
$$
   2\, \| u_k' (t) \|_H^2 + \partial_t \left( D (u_k (t)) - \| u_k (t) \|_H^2 \right)
 = 2\, \Re\, (f (t),u_k' (t))
$$
for almost all $t \in I$.

We then integrate the latter equality in $t \in I$ and use the Schwarz inequality,
obtaining
\begin{eqnarray*}
\lefteqn{
   2 \int_0^T \| u_k' (t) \|_H^2 dt + D (u_k (T)) + \| u_k (0) \|_H^2
}
\\
 & = &
   D (u_k (0)) + \| u_k (T) \|_H^2
 + 2\, \Re \int_0^T (f (t),u_k' (t))\, dt
\\
 & \leq &
   D (u_{0,k}) + \| u_k (T) \|_H^2
 + \int_0^T \left( \| f (t) \|_H^2 + \| u_k' (t) \|_H^2 \right) dt
\end{eqnarray*}
whence
\begin{equation}
\label{eq.bodiL2H}
   \int_0^T \| u_k' (t) \|_H^2 dt
 \leq
   C
 + D (u_{0,k})
 + \int_0^T \| f (t) \|_H^2\, dt,
\end{equation}
which is due to (\ref{eq.boinLiH}).

The basis $e_j$ used in the Galerkin method may be chosen so that $e_j \in S$
for all $j$ and we can take $u_{0,k}$ to be the projection in $S$ of $u_0$ on
the space spanned by $e_1, \ldots, e_k$.
Hence it follows that
   $u_{0,k}$ converges to $u_0$ strongly in $S$, as $k \to \infty$,
and
   $D (u_{0,k}) \leq D (u_0)$.
With these choices of the basis $e_j$ and initial data $u_{0,k}$ estimate (\ref{eq.bodiL2H})
shows that the sequence $u_k'$ belongs to a bounded set in $L^2 (I,H)$, and so
   $u' \in L^2 (I,H)$,
as desired.

Having disposed of this preliminary step, we come back to equalities
   (\ref{eq.basiclp}),
   (\ref{eq.mixedbc})
and apply the regularity theorem in the stationary case.
More precisely, we get
$$
\begin{array}{rcl}
   (\varepsilon \iD^m + \nu \iD) u + Ap
 & =
 & f - u',
\\
   A^\ast u
 & =
 & 0
\end{array}
$$
in $\cX$ and
   $u = 0$ up to order $m$ on the boundary of $\cX$ (if there is any)
for almost all $t \in I$.
Since $f - u' \in L^2 (I,H)$, we deduce that
   $u (t)$ belongs to $H^{2m} (\cX,F^i)$ and
   $p (t)$ belongs to $H^1 (\cX,F^{i-1})$.
Moreover, since
   $f (t) - u' (t) \mapsto (u (t),p (t))$
is a continuous linear mapping from $H$ into $H^{2m} (\cX,F^i) \times H^1 (\cX,F^{i-1})$,
it is clear that (\ref{eq.domStoc}) is satisfied.
\end{proof}

\section{Compactness theorems}
\label{s.compact}

The compactness theorems of this section go back at least as far as \cite{Lion69}.
They are included for the sake of selfcontained presentation.
The proofs are based on the following well-known lemma.

\begin{lemma}
\label{l.basicle}
Suppose that
   $\mathcal{B}_0$,
   $\mathcal{B}$ and
   $\mathcal{B}_1$
are three Banach spaces with the property that
   $\mathcal{B}_0 \hookrightarrow \mathcal{B} \hookrightarrow \mathcal{B}_1$,
the embedding of $\mathcal{B}$ into $\mathcal{B}_1$ being continuous and
the embedding of $\mathcal{B}_0$ into $\mathcal{B}$ being compact.
Then for each $\epsilon > 0$ there is a constant $c (\epsilon)$ depending on
$\epsilon$, such that
$$
   \| v \|_{\mathcal{B}}
 \leq
   \epsilon\, \| v \|_{\mathcal{B}_0} + c (\epsilon)\, \| v \|_{\mathcal{B}_1}
$$
for all $v \in \mathcal{B}_0$.
\end{lemma}

\begin{proof}
The proof is by contradiction.
Saying that the lemma is not true amounts to saying that there exists some $\epsilon > 0$
with the property that for each $c \in \R$ one has
$$
   \| v \|_{\mathcal{B}}
 \geq
   \epsilon\, \| v \|_{\mathcal{B}_0} + c\, \| v \|_{\mathcal{B}_1}
$$
for at least one $v \in \mathcal{B}_0$.
Taking $c = k$ we obtain a sequence of elements $v_k$ in $\mathcal{B}_0$ which satisfies
$$
   \| v_k \|_{\mathcal{B}}
 \geq
   \epsilon\, \| v_k \|_{\mathcal{B}_0} + k\, \| v_k \|_{\mathcal{B}_1}
$$
for all $k = 1, 2, \ldots$.
We consider then the normalised sequence
$
   \displaystyle
   w_k = \frac{v_k}{\| v_k \|_{\mathcal{B}_0}}
$
satisfying
\begin{equation}
\label{eq.wcttnsw}
   \| w_k \|_{\mathcal{B}}
 \geq
   \epsilon + k\, \| w_k \|_{\mathcal{B}_1}
\end{equation}
for all $k$.
Since $\| w_k \|_{\mathcal{B}_0} = 1$, the sequence $w_k$ is bounded in $\mathcal{B}$, and
so (\ref{eq.wcttnsw}) shows that
$
   \| w_k \|_{\mathcal{B}_1} \to 0,
$
as $k \to \infty$.
Furthermore, since the embedding of $\mathcal{B}_0$ into $\mathcal{B}$ is compact, the
sequence $w_k$ is relatively compact in $\mathcal{B}$.
Hence, we can extract from $w_k$ a subsequence $w_{k_j}$ strongly convergent in $\mathcal{B}$.
By the above, the limit of $w_{k_j}$ must be zero, however, this contradicts estimate
(\ref{eq.wcttnsw}) because
   $\| w_k \|_{\mathcal{B}} \geq \epsilon > 0$
for all $k$.
\end{proof}

Let
   $\mathcal{B}_0$,
   $\mathcal{B}$ and
   $\mathcal{B}_1$
be three Banach spaces such that
   $\mathcal{B}_0 \hookrightarrow \mathcal{B} \hookrightarrow \mathcal{B}_1$,
the embeddings being continuous.
We moreover assume that
   $\mathcal{B}_0$ and
   $\mathcal{B}_1$
are reflexive and the embedding of $\mathcal{B}_0$ into $\mathcal{B}$ is compact.
Given any real numbers $p_0$ and $p_1$ both greater than $1$, we consider the space
   $\mathcal{F}$
which consists of all functions $v \in L^{p_0} (I,\mathcal{B}_0)$ whose weak derivative
   $v' = \partial_t v$
belongs to $L^{p_1} (I,\mathcal{B}_1)$.
The space $\mathcal{F}$ is provided with the norm
$$
   \| v \|_{\mathcal{F}}
 = \| v \|_{L^{p_0} (I,\mathcal{B}_0)} + \| v' \|_{L^{p_1} (I,\mathcal{B}_1)},
$$
which makes it a Banach space.
It is obvious that
   $\mathcal{F} \subset L^{p_0} (I,\mathcal{B})$,
the embedding being continuous.
We are actually going to prove that this embedding is compact.

\begin{theorem}
\label{t.tptteic}
Under the above assumptions, the embedding of $\mathcal{F}$ into $L^{p_0} (I,\mathcal{B})$
is compact.
\end{theorem}

\begin{proof}
Let $u_k$ be a bounded sequence in $\mathcal{F}$.
We have to show that this sequence contains a subsequence which converges strongly in
$L^{p_0} (I,\mathcal{B})$.

Since the spaces $\mathcal{B}_0$ and $\mathcal{B}_1$ are reflexive and both $p_0$ and $p_1$
are not extreme, the spaces
   $L^{p_0} (I,\mathcal{B}_0)$ and
   $L^{p_1} (I,\mathcal{B}_1)$
are likewise reflexive, and hence $\mathcal{F}$ is reflexive.
Therefore, there exists a subsequence of $u_k$
   (for which we continue to write $u_k$ by abuse of notation)
such that
   $u_k$ converges to $u \in \mathcal{F}$ weakly in $\mathcal{F}$, as $k \to \infty$,
which means that
   $u_k$ converges to $u$ weakly in $L^{p_0} (I,\mathcal{B}_0)$
and
   $u_k'$ converges to $u'$ weakly in $L^{p_1} (I,\mathcal{B}_1)$.
It suffices to prove that
   $v_k = u_k - u$ converges to zero strongly in $L^{p_0} (I,\mathcal{B})$.

We shall have established the theorem if we prove that $v_k$ converges to zero strongly
in $L^{p_0} (I,\mathcal{B}_1)$.
In fact, due to Lemma \ref{l.basicle}, we get
$$
   \| v_k \|_{L^{p_0} (I,\mathcal{B})}
 \leq
   \epsilon\, \| v_k \|_{L^{p_0} (I,\mathcal{B}_0)}
 + c (\epsilon)\, \| v_k \|_{L^{p_0} (I,\mathcal{B}_1)}
$$
for all $k$.
Since the sequence $v_k$ is bounded in $\mathcal{F}$, it follows that
$$
   \| v_k \|_{L^{p_0} (I,\mathcal{B})}
 \leq
   \epsilon\, C
 + c (\epsilon)\, \| v_k \|_{L^{p_0} (I,\mathcal{B}_1)},
$$
with $C$ a constant independent of $k$.
On passing to the limit in this inequality we get
$$
   \limsup_{k \to \infty} \| v_k \|_{L^{p_0} (I,\mathcal{B})} \leq \epsilon\, C,
$$
for $v_k \to 0$ strongly in $L^{p_0} (I,\mathcal{B}_1)$.
Since $\epsilon > 0$ is arbitrary small in Lemma \ref{l.basicle}, the upper limit
is equal to zero, as desired.

To prove that $v_k$ converges to zero strongly in $L^{p_0} (I,\mathcal{B}_1)$, we
observe that $\mathcal{F}$ is embedded continuously into $C (\overline{I},\mathcal{B}_1)$.
This follows from Lemma \ref{l.primiti}, the continuity of the embedding is easy
to check.
We infer from this embedding that there is a constant $C > 0$ with the property that
$$
   \| v_k (t) \|_{\mathcal{B}_1} \leq c
$$
for all $t \in I$ and all $k$.
By the Lebesgue theorem, the strong convergence of $v_k$ to zero in $L^{p_0} (I,\mathcal{B}_1)$
is now proved if we show that $v_k (t) \to 0$ in the norm of $\mathcal{B}_1$ for almost all
$t \in I$.

Pick $t_0 \in I$.
Write
$$
   v_k (t_0) = v_k (t') - \int_{t_0}^{t'} v_k' (s) ds
$$
and by integration
$$
   v_k (t_0)
 = \frac{1}{t-t_0}
   \Big( \int_{t_0}^t v_k (t') dt' - \int_{t_0}^{t} dt' \int_{t_0}^{t'} v_k' (s) ds \Big)
$$
for all $t \in I$.
Hence it follows that $v_k (t_0) = I_k' + I_k''$, where
$$
\begin{array}{rcl}
   I_k'
 & =
 & \displaystyle
   \frac{1}{t-t_0}
   \int_{t_0}^t v_k (t') dt',
\\
   I_k''
 & =
 & \displaystyle
   \frac{-1}{t-t_0}
   \int_{t_0}^{t} (t-s) v_k' (s) ds.
\end{array}
$$
For a given $\epsilon > 0$, we choose $t$ in such a way that
$$
   \| I_k'' \|_{\mathcal{B}_1}
 \leq
   \int_{t_0}^{t} \| v_k' (s) \|_{\mathcal{B}_1} ds
 \leq
   \frac{\epsilon}{2}.
$$
Then, for this fixed $t$, we observe that $I_k' \to 0$ weakly in $\mathcal{B}_0$, and so
strongly in $\mathcal{B}_1$, as $k \to \infty$.
For sufficiently large $k$, we thus obtain
$
   \| I_k' \|_{\mathcal{B}_1} \leq \epsilon/2
$
whence
   $\| v_k (t_0) \|_{\mathcal{B}_1} \leq \epsilon$,
as desired.
\end{proof}

As usual, the most efficient compactness theorem can be formulated within the framework
of Hilbert spaces and it is based on the notion of fractional derivatives of a function.
Assume that
  $\mathcal{B}_0$,
   $\mathcal{B}$ and
   $\mathcal{B}_1$
are Hilbert spaces with continuous embeddings
   $\mathcal{B}_0 \hookrightarrow \mathcal{B} \hookrightarrow \mathcal{B}_1$,
where moreover the embedding of $\mathcal{B}_0$ into $\mathcal{B}$ is compact.
If $v$ is a function of $t \in \R$ with values in $\mathcal{B}_1$, we denote by $\hat{v}$
its Fourier transform, i.e.,
$$
   \hat{v} (\tau)
 = \int_{-\infty}^{+\infty} e^{- \imath \tau t} v (t) dt
$$
for $\tau \in \R$.
Given any $\lambda$ with $0 < \lambda \leq 1$, by the derivative of $v$ of order $\lambda$
in $t$ is meant the inverse Fourier transform of $(\imath \tau)^{\lambda} \hat{v}$, or
$$
   \widehat{\partial_t^{\lambda} v} (\tau)
 = (\imath \tau)^{\lambda} \hat{v} (\tau).
$$
Following \cite[p.~185]{Tema79} we define the space $\mathcal{H}^\lambda$ (depending on
   $\mathcal{B}_0$,
   $\mathcal{B}_1$ and
   $\lambda$)
to consist of all functions
   $v \in L^2 (\R,\mathcal{B}_0)$ such that
   $\partial_t^{\lambda} v \in L^2 (\R,\mathcal{B}_1)$.
This is a Hilbert space under the norm
$$
   \| v \|_{\mathcal{H}^\lambda}
 = \Big( \| v \|_{L^2 (\R,\mathcal{B}_0)}^2
       + \| |\tau|^{\lambda} \hat{v} \|_{L^2 (\R,\mathcal{B}_1)}^2
   \Big)^{1/2}.
$$
For a closed set $K \subset \R$, we denote by $\mathcal{H}^\lambda_K$ the (closed) subspace
of $\mathcal{H}^\lambda$ consisting of all functions with support in $K$.
The following theorem is contained in \cite[p.~186]{Tema79}.

\begin{theorem}
\label{t.Temam79}
Under the above hypotheses, if moreover $K$ is a compact subset of $\R$ and $\lambda > 0$,
the embedding of $\mathcal{H}^{\lambda}_K$ into $L^2 (\R,\mathcal{B})$ is compact.
\end{theorem}

\begin{proof}
Let $u_k$ be a bounded sequence in $\mathcal{H}^{\lambda}_K$.
We must show that $u_k$ contains a subsequence strongly convergent in $L^2 (\R,\mathcal{B})$.

Since $\mathcal{H}^{\lambda}$ is a Hilbert space, the sequence $u_k$ contains a subsequence
which converges weakly in this space.
To shorten notation, we continue to write $u_k$ for this subsequence.
It is clear that the limit function $u$ also belongs to $\mathcal{H}^{\lambda}$.
Therefore, on setting $v_k = u_k - u$ we see that $v_k$ is a bounded sequence in $\mathcal{H}^{\lambda}$,
which converges to zero weakly in $\mathcal{H}^{\lambda}$.
This means that
   $v_k$ converges to zero weakly in $L^2 (\R,\mathcal{B}_0)$, and
   $|\tau|^\lambda \hat{v}_k$ converges to zero weakly in $L^2 (\R,\mathcal{B}_1)$.
The theorem is proved if we show that $u_k$ converges to $u$ strongly in $L^2 (\R,\mathcal{B})$,
which is the same as $v_k \to 0$ strongly in $L^2 (\R,\mathcal{B})$.

Our next goal is to show that if $v_k$ converges to zero strongly in $L^2 (\R,\mathcal{B}_1)$
then the same is true in the norm of $L^2 (\R,\mathcal{B})$.
Due to Lemma \ref{l.basicle}, for each $\epsilon > 0$ there is a constant $c (\epsilon)$ such
that
$$
   \| v_k \|_{L^2 (\R,\mathcal{B})}
 \leq
   \epsilon\, \| v_k \|_{L^2 (\R,\mathcal{B}_0)}
 + c (\epsilon)\, \| v_k \|_{L^2 (\R,\mathcal{B}_1)}
$$
for all $k$.
Since the sequence $v_k$ is bounded in $L^2 (\R,\mathcal{B}_0$, it follows that
$$
   \| v_k \|_{L^2 (\R,\mathcal{B})}
 \leq
   \epsilon\, C
 + c (\epsilon)\, \| v_k \|_{L^2 (\R,\mathcal{B}_1)},
$$
with $C$ a constant independent of $k$.
On passing to the limit in this inequality we obtain
$$
   \limsup_{k \to \infty} \| v_k \|_{L^2 (\R,\mathcal{B})} \leq \epsilon\, C.
$$
Since $\epsilon > 0$ is arbitrary small in Lemma \ref{l.basicle}, the upper limit
must be equal to zero, as desired.

It remains to establish that $v_k$ converges to zero strongly in $L^2 (\R,\mathcal{B}_1)$.
To this end, write
$$
   I_k
 = \int_{- \infty}^{+ \infty} \| v_k (t) \|_{\mathcal{B}_1}^2 dt
 = \frac{1}{2 \pi} \int_{- \infty}^{+ \infty} \| \hat{v}_k (\tau) \|_{\mathcal{B}_1}^2 d \tau,
$$
which is due to the Plancherel theorem.
We have to show that $I_k$ tends to zero, as $k \to \infty$.
We get
\begin{eqnarray*}
   I_k
 & = &
   \frac{1}{2 \pi}
   \int_{|\tau| \leq R}
   \| \hat{v}_k (\tau) \|_{\mathcal{B}_1}^2
   d \tau
 + \frac{1}{2 \pi}
   \int_{|\tau| > R}
   \frac{\langle \tau \rangle^{2 \lambda}}{\langle \tau \rangle^{2 \lambda}}\,
   \| \hat{v}_k (\tau) \|_{\mathcal{B}_1}^2
   d \tau
\\
 & \leq &
   \frac{1}{2 \pi}
   \int_{|\tau| \leq R}
   \| \hat{v}_k (\tau) \|_{\mathcal{B}_1}^2
   d \tau
 + \frac{C}{\langle R \rangle^{2 \lambda}},
\end{eqnarray*}
for $v_k$ is bounded in $\mathcal{H}^{\lambda}$.
Here,
   $\langle \tau \rangle := (1 + |\tau|^2)^{1/2}$.
Given any $\epsilon > 0$, we choose $R$ sufficiently large, so that
$$
   \frac{C}{\langle R \rangle^{2 \lambda}}
 \leq
   \frac{\epsilon}{2}.
$$
Then
$$
   I_k
 \leq
   \frac{1}{2 \pi} \int_{|\tau| \leq R} \| \hat{v}_k (\tau) \|_{\mathcal{B}_1}^2 d \tau
 + \frac{\epsilon}{2},
$$
and the proof is completed by showing that, for this fixed $R$,
\begin{equation}
\label{eq.atpicbs}
   \int_{|\tau| \leq R} \| \hat{v}_k (\tau) \|_{\mathcal{B}_1}^2 d \tau
 \to
   0,
\end{equation}
as $k \to \infty$.
This is proved by the Lebesgue theorem.
If $\chi$ denotes the characteristic function of $K$, then $v_k = \chi v_k$ and
$$
   \hat{v}_k (\tau)
 = \int_{-\infty}^{+\infty} e^{- \imath \tau t} \chi (t) v_k (t) dt.
$$
Thus,
$$
   \| \hat{v}_k (\tau) \|_{\mathcal{B}_1}
 \leq
   \| e^{- \imath \tau t} \chi \|_{L^2 (\R)}
   \| v_k \|_{L^2 (\R,\mathcal{B}_1)}
$$
which is dominated by a constant independent of both $\tau$ and $k$.
On the other hand, if
   $\tau$ is fixed and
   $w$ an arbitrary element of $\mathcal{B}_0$,
then
$$
   (\hat{v}_k (\tau), w)_{\mathcal{B}_0}
 = \int_{-\infty}^{+\infty} (v_k (t), e^{- \imath \tau t} \chi (t) w)_{\mathcal{B}_0} dt,
$$
which tends to zero, as $k \to \infty$, for $v_k \to 0$ weakly in $L^2 (\R,\mathcal{B}_0)$.
It follows that the sequence $\hat{v}_k (\tau)$ converges to zero weakly in $\mathcal{B}_0$,
and so strongly in $\mathcal{B}$ and $\mathcal{B}_1$.
With this latter remark, the Lebesgue theorem on dominated convergence implies (\ref{eq.atpicbs}),
as desired.
\end{proof}

Using the methods of the last theorem, we are in a position to prove another compactness
theorem similar to Theorem \ref{t.tptteic}.
Nevertheless, this theorem is not contained in nor itself contains Theorem \ref{t.Temam79}.

\begin{theorem}
\label{t.tptteil}
Assume that
  $\mathcal{B}_0$,
   $\mathcal{B}$ and
   $\mathcal{B}_1$
are Hilbert spaces with continuous embeddings
   $\mathcal{B}_0 \hookrightarrow \mathcal{B} \hookrightarrow \mathcal{B}_1$,
the embedding of $\mathcal{B}_0$ into $\mathcal{B}$ being compact,
   and $I = (0,T)$ a finite interval.
Then the embedding of $\mathcal{F}$ with $p_0 = 2$ and $p_1 = 1$ into $L^2 (I,\mathcal{B})$
is compact.
\end{theorem}

\begin{proof}
Let $u_k$ be a bounded sequence in the space $\mathcal{F}$.
Denote by $v_k$ the function defined on all of $\R$, which is equal to $u_k$ on $[0,T]$
and to zero outside this interval.
By Theorem \ref{t.Temam79}, we shall have established the theorem if we show that the
sequence $v_k$ remains bounded in the space $\mathcal{H}^\lambda$ with some $\lambda > 0$.

By Lemma \ref{l.primiti}, each function $u_k$ is, after possible modification on a set of
measure zero, continuous on $[0,T]$ with values in $\mathcal{B}_1$.
More precisely, the embedding of $\mathcal{F}$ into $C (\overline{I},\mathcal{B}_1)$ is
continuous.
Since $v_k$ has two discontinuities at $0$ and $T$, the distribution derivative of $v_k$
is given by
$$
   v_k' = w_k + u_k (0) \delta_0 - u_k (T) \delta_T,
$$
where
   $\delta_0$ and $\delta_T$ are the Dirac distributions at $0$ and $T$,
and
   $w_k$ the derivative of $u_k$ on $[0,T]$ extended by zero to the complement of $[0,T]$.
After a Fourier transformation we get
\begin{equation}
\label{eq.aaFtwge}
   (\imath \tau)\, \hat{v}_k (\tau)
 = \hat{w}_k (\tau) + u_k (0) - u_k (T)\, \exp (- \imath T \tau)
\end{equation}
for all $\tau \in \R$.

Since the functions $u_k'$ remain bounded in $L^1 (I,\mathcal{B}_1)$, the functions $w_k$
are bounded in $L^1 (\R,\mathcal{B}_1)$ and the functions $\hat{w}_k$ are bounded in the
space $L^{\infty} (\R,\mathcal{B}_1)$, i.e.,
$$
   \| \hat{w}_k (\tau) \|_{\mathcal{B}_1} \leq C
$$
for all $\tau \in \R$ and all $k$.
As mentioned, the embedding of $\mathcal{F}$ into $C (\overline{I},\mathcal{B}_1)$ is
continuous whence
$$
\begin{array}{rcl}
   \| u_k (0) \|_{\mathcal{B}_1}
 & \leq
 & c,
\\
   \| u_k (T) \|_{\mathcal{B}_1}
 & \leq
 & c
\end{array}
$$
and (\ref{eq.aaFtwge}) implies that
$$
   |\tau|^2\, \| \hat{v}_k (\tau) \|_{\mathcal{B}_1}^2
 \leq
   (C + 2 c)^2
$$
for all $\tau \in \R$ and all $k$.

Fix any positive $\lambda$ satisfying
   $\displaystyle \lambda < \frac{1}{2}$.
From the inequality
$$
   |\tau|^{2 \lambda}
 \leq
   c (\lambda)\, \frac{1 + |\tau|^2}{1 + |\tau|^{2 (1-\lambda)}}
$$
for all $\tau \in \R$,
   the constant $c (\lambda)$ depending on $\lambda$,
it follows that
\begin{eqnarray*}
   \int_{-\infty}^{+\infty}
   |\tau|^{2 \lambda}\, \| \hat{v}_k (\tau) \|_{\mathcal{B}_1}^2
   d \tau
 & \leq &
   c (\lambda)
   \int_{-\infty}^{+\infty}
   \frac{1 + |\tau|^2}{1 + |\tau|^{2 (1-\lambda)}}\,
   \| \hat{v}_k (\tau) \|_{\mathcal{B}_1}^2
   d \tau
\\
 & \leq &
   c (\lambda)
   \Big(
   \int_{-\infty}^{+\infty}
   \| \hat{v}_k (\tau) \|_{\mathcal{B}_1}^2
   d \tau
 + \int_{-\infty}^{+\infty}
   \frac{(C + 2 c)^2}{1 + |\tau|^{2 (1-\lambda)}}\,
   d \tau
   \Big).
\end{eqnarray*}
By the Plancherel theorem, we see that
$$
   \int_{-\infty}^{+\infty}
   \| \hat{v}_k (\tau) \|_{\mathcal{B}_1}^2
   d \tau
 = \int_0^T
   \| v_k (t) \|_{\mathcal{B}_1}^2
   dt
$$
and these integrals are bounded uniformly in $k$.
On the other hand, since
   $\displaystyle \lambda < \frac{1}{2}$,
the integral
$$
   \int_{-\infty}^{+\infty}
   \frac{1}{1 + |\tau|^{2 (1-\lambda)}}\,
   d \tau
$$
is convergent.
Summarising we conclude that
$$
   \int_{-\infty}^{+\infty}
   |\tau|^2\, \| \hat{v}_k (\tau) \|_{\mathcal{B}_1}^2
   d \tau
 \leq
   C,
$$
where $C$ depends on $\lambda$.
It is now clear that the sequence $v_k$ is bounded in $\mathcal{H}^\lambda$ and each
$v_k$ has support in $[0,T]$.
\end{proof}

Assuming only that
   $\mathcal{B}_0$ and $\mathcal{B}$ are Banach spaces and
   $\mathcal{B}_1$ a Hilbert space,
it can be proved in a similar way that the embedding of $\mathcal{F}$ with
   $1 < p_0 < \infty$ and
   $p_1 = 1$
into $L^{p_0} (I,\mathcal{B})$ is compact.

\section{Existence of a weak solution}
\label{s.exiweak}

This section is concerned with existence theorem for a weak solution of the regularised
Navier-Stokes equations.
We start with a weak formulation of these equations, following \cite{Lera34a}, and
state an existence theorem for such solution.
The proof of this theorem is due to \cite{Lion69}, see also \cite{Lady70}.
It is based on the construction of an approximate solution by the Galerkin method,
followed by a passage to the limit using, in particular, an a priori estimate for a
fractional derivative in time of the approximate solution and a compactness theorem
contained in Section \ref{s.compact}.
Then we develop the uniqueness theorem for weak solutions provided that $\varepsilon > 0$.
If $\varepsilon = 0$, the arguments still go in the case $n \leq 2$.
In the case of $n > 2$ there is a gap between the class of functions where existence
is known, and the smaller classes where uniqueness is proved.
For $\varepsilon > 0$, we show the existence of more regular solutions, assuming more
regularity on the data.
A similar result holds in the conventional case $\varepsilon = 0$ for local solutions,
i.e., those defined on a ``small'' time interval.

Recall that we define on $H^m_0 (\cX,F^i)$, and in particular on $S$, a trilinear
form by setting
$$
   (N (u,v),w) = \int_{\cX} (T (v) u, w)_x dx.
$$

\begin{lemma}
\label{l.(N(uv)w)}
The form $(N (u,v),w)$ is defined and trilinear continuous on the Cartesian product
   $H^m_0 (\cX,F^i) \times H^m_0 (\cX,F^i) \times H^m_0 (\cX,F^i)$,
if $m \geq (n+2)/4$, $n$ being the dimension of $\cX$.
\end{lemma}

\begin{proof}
We use here a well-known result on pointwise multiplication of functions in Sobolev
spaces.
Namely, let $s$ and $s_1$, $s_2$ be real numbers satisfying
   $s_1 + s_2 \geq 0$.
   $s_1, s_2 \geq s$ and
   $s_1 + s_2 -  s > n/2$,
where the strictness of the last two inequalities can be interchanged if $s = 0, 1, \ldots$.
Then the pairwise multiplication of functions extends to a continuous bilinear mapping
   $H^{s_1} \times H^{s_2} \to H^s$.
We apply this result with $s = 0$ and $s_1 = s_2 = m$.
The condition $s_1 + s_2 -  s \geq n/2$ is obviously fulfilled, for $m \geq (n+2)/4$.
In coordinate patches on $\cX$, over which the bundle $F^i$ is trivial, we get
$$
   (T (v) u, w)_x
 = \sum_{j = 1, \ldots, k_i \atop
         k = 1, \ldots, k_i}
   T_{j k} (v) u_k \overline{w}_j,
$$
where
   $k_i$ is the rank of $F^i$,
   $u_k$ and $v_j$ are coordinates of $u$ and $v$ relative to a local frame of $F^i$,
   respectively,
and
   $T_{j k}$ are first order scalar differential operators in local coordinates.
Hence,
\begin{eqnarray*}
   | (N (u,v), w) |
 & \leq &
   c\,
   \| Tv \|_{L^2 (\cX,\mathrm{Hom}\, F^i)} \| u \|_{H^m (\cX,F^i)} \| w \|_{H^m (\cX,F^i)}
\\
 & \leq &
   c\, \sqrt{D (u)} \sqrt{D (v)} \sqrt{D (w)}
\end{eqnarray*}
for all $u, v, w \in H^m_0 (\cX,F^i)$, where
   $c$ is a constant independent of $u$, $v$ and $w$ and not necessarily the same in diverse
   applications.
This estimate ensures the continuity of the trilinear form.
\end{proof}

In particular, the form $(N (u,v),w)$ is well defined and trilinear continuous on
   $S \times S \times S$,
if $m \geq (n+2)/4$.

If
   $A^\ast u = 0$ in $\cX$ and
   $n (u) = 0$ on the boundary of $\cX$,
then
   $(N (u,v),v) = 0$ for all $v \in H^m_0 (\cX,F^i)$,
which is due to Corollary \ref{c.algestr}.
For $u, v \in S$, the section $N (u,v)$ can be thought of as a continuous linear
functional on $S$.
Set $N (u) = N (u,u) \in S'$ for $u \in S$.

In the classical formulation the first mixed problem for the generalised Navier-Stokes
equations consists in finding sections
   $u$ and
   $p$
of $F^i$ and $F^{i-1}$ over $\cC$, respectively, such that
\begin{equation}
\label{eq.classNS}
\begin{array}{rcl}
   \partial_t u + (\varepsilon \iD^m + \nu \iD) u + N (u) + Ap
 & =
 & f,
\\
   A^\ast u
 & =
 & 0
 \end{array}
\end{equation}
in the cylinder $\cC$ and
\begin{equation}
\label{eq.classmc}
\begin{array}{rclcl}
   u (\cdot,0)
 & =
 & u_0
 & \mbox{on}
 & \cX,
\\
   u
 & =
 & 0
 & \mbox{up to order $m$ on}
 & \partial \cX \times I,
 \end{array}
\end{equation}
where
   $f$ and $u_0$ are given sections of $F^i$ over the cylinder and its bottom $\cX$
respectvely.
Conditions (\ref{eq.classmc}) specify the first mixed problem for evolution equations
(\ref{eq.classNS}).

Let $(u,p)$ be a classical solution of problem (\ref{eq.classNS}), (\ref{eq.classmc}),
say
   $u \in C^{2m} (\overline{\cC},F^i)$ and
   $p \in C^1 (\overline{\cC},F^{i-1})$.
Then one can check easily that
$$
   (u'_t, v) + (Vu, v) + (N (u), v) = (f, v)
$$
for each element $v$ of $\mathcal{S}$.
By continuity, this equation holds also for each $v \in S$.
Since
   $(u'_t, v) = \partial_t\, (u,v)$,
this suggests the following weak formulation of problem (\ref{eq.classNS}), (\ref{eq.classmc}),
cf. \cite{Lera33,Lera34a,Lera34b}.
Given sections
   $f \in L^2 (I,S')$ and
   $u_0 \in H$,
find $u \in L^2 (I,S)$ such that
\begin{equation}
\label{eq.NSweakf}
\begin{array}{rclcl}
   \partial_t\, (u,v) + (Vu,v) + (N (u),v)
 & =
 & (f,v)
 & \mbox{for all}
 & v \in S,
\\
   u (0)
 & =
 & u_0.
 &
 &
 \end{array}
\end{equation}
(By (\ref{eq.Stokes}), any solution to this problem leads to a solution in some weak
 sense of problem (\ref{eq.classNS}), (\ref{eq.classmc}).)

If a section $u$ merely belongs to $L^2 (I,S)$, the condition $u (0) = u_0$ need not make
sense.
But if $u$ belongs to $L^2 (I,S)$ and satisfies the variational equation of (\ref{eq.NSweakf})
then we can argue as in the linear case (using Lemma \ref{l.primiti}) to see that $u$ is
equal a.e. to some continuous function, so that $u (0) = u_0$ is meaningful.
Before showing this, we prove a lemma.

\begin{lemma}
\label{l.lapwtsB}
Suppose that $u \in L^2 (I,S)$.
Then the function $N (u)$ defined by
   $(N (u) (t),v) = (N (u (t),u (t)),v)$,
for almost all $t \in I$ and all $v \in S$, belongs to $L^1 (I,S')$.
\end{lemma}

\begin{proof}
For almost all $t \in I$, the value $N (u) (t)$ is an element of $S'$, and the measurability
of the function $N (u)$ of $t \in I$ with values in $S'$ is easily verified.
Moreover, since $(N (u,v),w)$ is a continuous trilinear form on the space $S$, we conclude
immediately that
$$
   \| N (v) \|_{S'} \leq c\, D (v)
$$
for all $v \in S$.
Therefore,
$$
   \int_0^T \| N (u) (t) \|_{S'} dt
 \leq
   c\, \int_0^T D (u) (t)\, dt
 < \infty,
$$
and the lemma is proved.
\end{proof}

Now, if $u \in L^2 (I,S)$ satisfies (\ref{eq.NSweakf}), then, according to the above
lemma, one can write (\ref{eq.NSweakf}) as
$$
   \partial_t\, (u,v)
 = (f - Vu - N (u),v)
$$
for all $v \in S$.
Since $Vu$ belongs to $L^2 (I,S')$, as in the linear case, the function $f - Vu - N (u)$
belongs to $L^1 (I,S')$.
Lemma \ref{l.primiti} implies that
   $u' \in L^1 (I,S')$ and
   $u' = f - Vu - N (u)$,
and that $u$ is equal a.e. to a continuous function of $t \in \overline{I}$ with values
in $S'$.
This remark makes the initial condition $u (0) = u_0$ meaningful, as desired.

An alternative formulation of problem (\ref{eq.NSweakf}) reads as follows.
Let
   $f \in L^2 (I,S')$ and
   $u_0 \in H$
be given sections.
Find a section
   $u \in L^2 (I,S)$ such that
   $u' \in L^1 (I,S')$
and
\begin{equation}
\label{eq.NSpowif}
\begin{array}{rclcl}
   u' + Vu + N (u)
 & =
 & f
 & \mbox{a.e. on}
 & I,
\\
   u (0)
 & =
 & u_0.
 &
 &
 \end{array}
\end{equation}
We have shown that any solution of problem (\ref{eq.NSweakf}) is a solution of
problem (\ref{eq.NSpowif}).
The converse is very easily checked and these problems are thus equivalent.
The existence of solutions of these problems is ensured by our next theorem, cf.
Theorem 3.1 of \cite[p.~191]{Tema79}.

\begin{theorem}
\label{t.exnonNS}
Assume that
   $f \in L^2 (I,S')$ and
   $u_0 \in H$
are arbitrary data.
Then there exists at least one function $u \in L^2 (I,S)$ satisfying
   $u' \in L^1 (I,S')$ and
   (\ref{eq.NSpowif}).
Moreover,
   $u \in L^\infty (I,H)$ and
   $u$ is a weakly continuous function of $t \in \overline{I}$ with values in $H$.
\end{theorem}

Note that by the weak continuity of $u$ is meant that, for each $v \in H$, the
function $t \mapsto (u (t),v)$ is continuous on $I$.
The weak continuity result is a direct consequence of
   $u \in L^\infty (I,H)$,
   the continuity of $u$ in $t \in \overline{I}$ with values in $S'$,
and of
   Lemma 1.4 of \cite[p.~178]{Tema79}.

\begin{proof}
\textit{Approximate solution}.
We apply the Galerkin procedure as in the linear case.
Since
   $S$ is separable and
   $\mathcal{S}$ is dense in $S$,
there is a sequence of linearly independent elements
   $( e_i )_{i = 1, 2, \ldots}$
of $\mathcal{S}$, which is complete in $S$.
(It should be noted that the $e_i$ are chosen in $\mathcal{S}$ for simplicity.
 With some technical modifications we could take these elements in $S$.)
For every $k = 1, 2, \ldots$ we define an approximate solution of problem
   (\ref{eq.NSweakf})
by
$$
   u_k = \sum_{i=1}^k c_{k,i} (t) e_i
$$
and
\begin{equation}
\label{eq.appsNSc}
\begin{array}{rclcl}
   (u_k',e_j) + (V u_k,e_j) + (N (u_k),e_j)
 & =
 & (f,e_j)
 & \mbox{for all}
 & j = 1, \ldots, k,
\\
   u_k (0)
 & =
 & u_{0,k},
 &
 &
 \end{array}
\end{equation}
where $u_{0,k}$ is the orthogonal projection in $H$ of $u_0$ onto the space spanned by
the elements $e_1, \ldots, e_k$.
(We could take for $u_{0,k}$ any element of the space spanned by $e_1, \ldots, e_k$,
 such that $u_{0,k} \to u_0$ in the norm of $H$, as $k \to \infty$.)

Equations (\ref{eq.appsNSc}) form a system of nonlinear ordinary differenial equations for
the unknown scalar-valued functions $c_{k,i}$ on the interval $I$, where $1 \leq i \leq k$.
Indeed, we get
$$
   \sum_{i=1}^k (e_i,e_j)\, c_{k,i}' (t)
 + \sum_{i=1}^k (V e_i,e_j)\, c_{k,i} (t)
 + \sum_{l = 1, \ldots, k \atop
         m = 1, \ldots, k} (N (e_l,e_m),e_j)\, c_{k,l} (t) c_{k,m} (t)
 = (f,e_j)
$$
for all $j = 1, \ldots, k$.
On inverting the regular matrix with entries
   $(e_i,e_j)$,
where $1 \leq i, j \leq k$, we rewrite the differential equations in the usual form
\begin{equation}
\label{eq.nfoeftc}
   c_{k,i}' (t)
 + \sum_{j=1}^k a_{i,j}\, c_{k,j} (t)
 + \sum_{l = 1, \ldots, k \atop
         m = 1, \ldots, k} a_{i,l,m}\, c_{k,l} (t) c_{k,m} (t)
 = \sum_{j=1}^k b_{i,j}\, (f,e_j)
\end{equation}
for $i = 1, \ldots, k$, where
   $a_{i,j}$,
   $a_{i,l,m}$ and
   $b_{i,j}$
are constant complex numbers.
Furthermore, the condition $u_k (0) = u_{0,k}$ is equivalent to $k$ equations
   $c_{k,i} (0) = \pi_i (u_{0,k})$,
where $\pi_i$ is the projection of $u_{0,k}$ on the one-dimensional space spanned by
the element $e_i$.

The system of nonlinear equations (\ref{eq.nfoeftc}) with initial conditions
   $c_{k,i} (0) = \pi_i (u_{0,k})$
has a maximal solution defined on some interval $[0,T (k)]$.
If $T (k) < T$, then $\|u_k (t) \|_{H}$ must tend to $+ \infty$ as $t \to T (k)$.
However, the a priori estimates we go to prove show that this does not happen, and so
$T (k) = T$, showing that the coefficients $c_{k,i} (t)$ are defined on all of $[0,T]$.

\textit{A priori estimates}.
The first a priori estimates are established as in the linear case.
We multiply the first equation of (\ref{eq.appsNSc}) by $c_{k,j} (t)$ and sum up these
equations for $j = 1, \ldots, k$.
From the properties of the nonlinear term we deduce that $(N (u_k),u_k) = 0$ whence
$$
   (u_k' (t),u_k (t)) + (V u_k (t),u_k (t))
 = (f (t),u_k (t)).
$$
Since
$
   \partial_t (u_k (t),u_k (t)) = 2\, \Re\, (u_k' (t),u_k (t)),
$
we write
\begin{eqnarray*}
   \partial_t\, \| u_k (t) \|_H^2 + 2 \left( D (u_k (t)) - \| u_k (t) \|_H^2 \right)
 & = &
   2\, \Re\, (f (t),u_k (t))
\\
 & \leq &
   2\, \| f (t) \|_{S'} \| u_k (t) \|_S
\\
 & \leq &
   c\, \| f (t) \|_{S'}^2 + \frac{1}{c}\, D (u_k (t))
\end{eqnarray*}
with arbitrary constant $c > 0$.
We choose $c$ in such a way that $2 - 1/c > 0$.
On the other hand,
$$
   \partial_t\, \| u_k (t) \|_H^2 - 2 \| u_k (t) \|_H^2
 = e^{2t}\, \partial_t \left( e^{- 2t} \| u_k (t) \|_H^2 \right),
$$
as is easy to check.
Therefore,
\begin{equation}
\label{eq.tlhsooa}
   \partial_t \left( e^{- 2t} \| u_k (t) \|_H^2 \right)
 + \Big( 2 - \frac{1}{c} \Big) e^{- 2t}\, D (u_k (t))
 \leq
   c\, e^{- 2t}\, \| f (t) \|_{S'}^2.
\end{equation}
Integrating the inequality from $0$ to $t$, where $t$ is a fixed point of $I$, we obtain,
in particular,
\begin{eqnarray*}
   \| u_k (t) \|_H^2
 & \leq &
   e^{2t}\, \| u_{0,k} \|_H^2
 + c \int_0^t e^{2 (t-t')} \| f (t') \|_{S'}^2 dt'
\\
 & \leq &
   e^{2t}\, \| u_0 \|_H^2
 + c \int_0^t e^{2 (t-t')} \| f (t') \|_{S'}^2 dt'
\end{eqnarray*}
for almost all $t \in I$.
Therefore,
$$
   \sup_{t \in I} \| u_k (t) \|_H^2
 \leq
   e^{2T}
   \Big( \| u_0 \|_H^2 + c \int_0^T \| f (t) \|_{S'}^2 dt \Big).
$$
the right-hand side being finite and independent of $k$.
We have thus proved that the sequence $u_k$ remains in a bounded subset of $L^\infty (I,H)$,
i.e., there is a constant $C$ such that
\begin{equation}
\label{eq.uboasiH}
   \| u_k \|_{L^\infty (I,H)} \leq C
\end{equation}
for all $k$.

Furthermore, on integrating inequality (\ref{eq.tlhsooa}) in $t$ over all of $I$ we arrive at
the estimates
\begin{eqnarray*}
   \| u_k (T) \|_H^2
 + \Big( 2 - \frac{1}{c} \Big)
   \int_0^T \!\!\! e^{2 (T-t)} D (u_k (t)) dt
 \!\! & \!\! \leq \!\! & \!\!
   e^{2T} \| u_{0,k} \|_H^2
 + c \int_0^T \!\!\! e^{2 (T-t)} \| f (t) \|_{S'}^2 dt
\\
 \!\! & \!\! \leq \!\! & \!\!
   e^{2T}\, \| u_0 \|_H^2
 + c \int_0^T \!\!\! e^{2 (T-t)} \| f (t) \|_{S'}^2 dt
\end{eqnarray*}
for each $k = 1, 2, \ldots$.
This shows that the sequence $u_k$ remains in a bounded subset of $L^2 (I,S)$, i.e.,
\begin{equation}
\label{eq.uboasiS}
   \| u_k \|_{L^2 (I,S)} \leq C
\end{equation}
for all $k$, with $C$ a constant independent of $k$.

\textit{Estimates for fractional derivative}.
Let $v_k$ denote the function of $t \in \R$ with values in $S$, which is equal to $u_k$ on
$[0,T]$ and to zero in the complement of this interval.
The Fourier transform of $v_k$ is denoted by $\hat{v}_k$.
In addition to the previous inequalities, which are similar to the estimates in the linear case,
we want to show that
\begin{equation}
\label{eq.eofdovk}
   \int_{-\infty}^{+\infty} |\tau|^{2 \lambda} \| \hat{v}_k (\tau) \|_{H}^2 d \tau \leq C
\end{equation}
for some $\lambda > 0$, both constants $\lambda$ and $C$ being independent of $k$.
Along with (\ref{eq.uboasiS}) this will imply that $v_k$ belongs to a bounded set of
   $\mathcal{H}^\lambda$
corresponding to
   $\mathcal{B}_0 = S$ and
   $\mathcal{B}_1 = H$,
and this will enable us to apply the compactness result of Theorem \ref{t.Temam79}.

In order to prove (\ref{eq.eofdovk}) we observe that equations (\ref{eq.appsNSc}) can be written
in the form
$$
   \partial_t\, (v_k, e_j)
 = (w_k, e_j) + (u_{0,k}, e_j)\, \delta_0 - (u_k (T), e_j)\, \delta_T
$$
for all $j = 1, \ldots, k$, where
   $w_k$ is equal to $f - V u_k - N (u_k)$ on $[0,T]$ and
   zero outside this interval.
(Compare this with the proof of Theorem \ref{t.tptteil}.)
On applying the Fourier transform we get
\begin{equation}
\label{eq.oatFtwg}
   (\imath \tau)\, (\hat{v}_k (\tau), e_j)
 = (\hat{w}_k (\tau), e_j) + (u_{0,k}, e_j) - (u_k (T), e_j)\, \exp (- \imath T \tau)
\end{equation}
for all $\tau \in \R$.

We multiply (\ref{eq.oatFtwg}) by (the complex conjugate of) the Fourier transform of
the function equal to $c_{k,j} (t)$ for $t \in [0,T]$ and zero for $t \in \R \setminus [0,T]$,
and sum up the resulting equalities for $j = 1, \ldots, k$.
This yields
\begin{equation}
\label{eq.eftFtovk}
   (\imath \tau)\, \| \hat{v}_k (\tau) \|_H^2
 = (\hat{w}_k (\tau), \hat{v}_k (\tau))
 + (u_{0,k}, \hat{v}_k (\tau))
 - (u_k (T), \hat{v}_k (\tau))\, \exp (- \imath T \tau).
\end{equation}
By inequality $\| N (v) \|_{S'} \leq c\, D (v)$ for all $v \in S$ (see the proof
of Lemma \ref{l.lapwtsB}), we obtain
$$
   \int_0^T \| w_k (t) \|_{S'} dt
 \leq
   \int_0^T \left( \| f (t) \|_{S'} + \sqrt{D (u_k (t))} + c D (u_k (t)) \right) dt,
$$
and the right-hand side remains bounded uniformly in $k$ according to (\ref{eq.uboasiS}).
It follows that
$$
   \sup_{\tau \in \R} \| \hat{w}_k (\tau) \|_{S'}
 \leq C
$$
for all $k$, the constant $C$ being different in diverse applications.

By (\ref{eq.uboasiH})
$$
\begin{array}{rcl}
   \| u_k (0) \|_H & \leq & c,
\\
   \| u_k (T) \|_H & \leq & c,
\end{array}
$$
and so we deduce from (\ref{eq.eftFtovk}) that
$$
   |\tau|\, \| \hat{v}_k (\tau) \|_H^2
 \leq
   C \sqrt{D (\hat{v}_k (\tau))} + 2c\, \| \hat{v}_k (\tau) \|_H
$$
or
$$
   |\tau|\, \| \hat{v}_k (\tau) \|_H^2
 \leq
   (C + 2 c)\, \sqrt{D (\hat{v}_k (\tau))}
$$
for all $\tau \in \R$.

Choose any positive $\lambda$ satisfying
   $\displaystyle \lambda < \frac{1}{4}$.
From the inequality
$$
   |\tau|^{2 \lambda}
 \leq
   c (\lambda)\, \frac{1 + |\tau|}{1 + |\tau|^{1 - 2 \lambda}}
$$
for all $\tau \in \R$,
   the constant $c (\lambda)$ depending on $\lambda$,
it follows that
\begin{eqnarray*}
\lefteqn{
   \int_{-\infty}^{+\infty}
   |\tau|^{2 \lambda}\, \| \hat{v}_k (\tau) \|_{H}^2
   d \tau
}
\\
 & \leq &
   c (\lambda)
   \int_{-\infty}^{+\infty}
   \frac{1 + |\tau|}{1 + |\tau|^{1-2 \lambda}}\,
   \| \hat{v}_k (\tau) \|_{H}^2
   d \tau
\\
 & \leq &
   c (\lambda)
   \Big(
   \int_{-\infty}^{+\infty}
   \| \hat{v}_k (\tau) \|_{H}^2
   d \tau
 + (C + 2 c)
   \int_{-\infty}^{+\infty}
   \frac{\sqrt{D (\hat{v}_k (\tau))}}{1 + |\tau|^{1-2 \lambda}}\,
   d \tau
   \Big).
\end{eqnarray*}
By the Plancherel theorem, we see that
$$
   \int_{-\infty}^{+\infty}
   \| \hat{v}_k (\tau) \|_{H}^2
   d \tau
 = \int_0^T
   \| v_k (t) \|_{H}^2
   dt
$$
and these integrals are bounded uniformly in $k$.
Hence, we shall have established (\ref{eq.eofdovk}) if we show that the integral
$$
   \int_{-\infty}^{+\infty}
   \frac{\sqrt{D (\hat{v}_k (\tau))}}{1 + |\tau|^{1-2 \lambda}}\,
   d \tau
$$
is bounded uniformly in $k$.
By the Schwarz inequality and the Plancherel theorem we can bound these integrals by
$$
   \Big( \int_{-\infty}^{+\infty} \frac{d \tau}{(1 + |\tau|^{1-2 \lambda})^2}\, \Big)^{1/2}
   \Big( \int_{-\infty}^{+\infty} D (v_k (t))\, dt \Big)^{1/2},
$$
the first factor being finite, for $\lambda < 1/4$, and
the second factor is bounded uniformly in $k$ by (\ref{eq.uboasiS}).
The proof of (\ref{eq.eofdovk}) is complete.

\textit{Passage to the limit}.
Estimates (\ref{eq.uboasiH}) and (\ref{eq.uboasiS}) enable us to assert that there is a
subsequence of $u_k$ which converges to an element
   $u \in L^\infty (I,H) \cap L^2 (I,S)$
both in the weak$^\ast$ topology of $L^\infty (I,H)$ and in the weak topology of $L^2 (I,S)$.
It will cause no confusion if we continue to write $u_k$ for this subsequence.
By (\ref{eq.eofdovk}) and the compactness result of Theorem \ref{t.Temam79} there is no
loss of generality in assuming that $u_k$ converges to $u$ strongly in $L^2 (I,H)$.
The former of the convergence results allows us to pass to the limit.
We argue essentially in the same way as in the linear case.

Let $\phi$ be a continuously differentiable scalar-valued function of $t \in [0,T]$
satisfying $\phi (T) = 0$.
We multiply (\ref{eq.appsNSc}) by $\phi (t)$ and then integrate by parts.
This leads to the equations
\begin{eqnarray*}
\lefteqn{
 - \int_0^T \!\! (u_k (t), e_j) \phi' (t) dt
 + \int_0^T \!\! (V u_k (t), e_j) \phi (t) dt
 + \int_0^T \!\! (N (u_k (t)), e_j) \phi (t) dt
}
\\
 & = &
   (u_{0,k}, e_j)\, \phi (0)
 + \int_0^T \!\! (f (t), e_j) \phi (t) dt
   \hspace{4cm}
\end{eqnarray*}
for $j = 1, \ldots, k$.
Passing to the limit with the subsequence of $u_k$ is easy for the linear terms.
For the nonlinear term we apply Lemma \ref{l.ptlfnlt} to be proved below.
In the limit we find that the equality
\begin{eqnarray}
\label{eq.ptlfthe}
\lefteqn{
 - \int_0^T \!\! (u (t), v) \phi' (t) dt
 + \int_0^T \!\! (V u (t), v) \phi (t) dt
 + \int_0^T \!\! (N (u (t)), v) \phi (t) dt
}
\nonumber
\\
 & = &
   (u_{0}, v)\, \phi (0)
 + \int_0^T \!\! (f (t), v) \phi (t) dt
   \hspace{4cm}
\nonumber
\\
\end{eqnarray}
holds for $v = e_1, e_2, \ldots$.
Hence it follow that the equality is valid for any element $v \in S$ which is a finite
linear combination of the $e_k$.
And by the continuity argument (\ref{eq.ptlfthe}) is still true for any $v \in S$.
In particular, on writing (\ref{eq.ptlfthe}) with arbitrary functions
   $\phi \in C^\infty_{\mathrm{comp}} (I)$
we conclude readily that $u$ satisfies (\ref{eq.NSweakf}) in the sense of distributions.

Finally, it remains to prove that $u$ satisfies $u (0) = u_0$.
For this purpose we multiply (\ref{eq.NSweakf}) by $\phi$ and integrate in $t \in I$.
After integrating the first term by parts we get
\begin{eqnarray*}
\lefteqn{
 - \int_0^T \!\! (u (t),v)\, \phi' (t) dt
 + \int_0^T \!\! (V u (t),v)\, \phi (t) dt
 + \int_0^T \!\! (N (u (t)),v)\, \phi (t) dt
}
\\
 & = &
   (u_{0}, v)\, \phi (0)
 + \int_0^T \!\! (f (t),v)\, \phi (t) dt
   \hspace{4cm}
\end{eqnarray*}
for all $v \in S$.
On comparing this with (\ref{eq.ptlfthe}) we see that
$
   (u (0) - u_0, v)\, \phi (0) = 0
$
for all $v \in S$ and for each function $\phi$ of the type considered.
We can choose $\phi$ with $\phi (0) = 1$.
Thus,
$
   (u (0) - u_0, v) = 0
$
for all $v \in S$.
This equality implies $u (0) = u_0$, as desired.

The proof of Theorem \ref{t.exnonNS} will be complete once we show the lemma mentioned
above.

\begin{lemma}
\label{l.ptlfnlt}
Assume that
   $u_k$ converges to $u$ weakly in $L^2 (I,S)$ and strongly in $L^2 (I,H)$.
Then
$$
   \int_0^T (N (u_k (t)), w (t))\, dt
 \to
   \int_0^T (N (u (t)), w (t))\, dt
$$
for any section $w \in C^1 (\overline{\cC},F^i)$.
\end{lemma}

\begin{proof}
We write
\begin{eqnarray*}
   \int_0^T (N (u_k (t)), w (t))\, dt
 & = &
 - \int_0^T (N (u_k (t), w (t)), u_k (t))\, dt
\\
 & = &
 - \int_0^T (T (w (t)) u_k (t), u_k (t))\, dt,
\end{eqnarray*}
which is due to Corollary \ref{c.algestr}.
The right-hand side converges to
\begin{eqnarray*}
 - \int_0^T (T (w (t)) u (t), u (t))\, dt
 & = &
 - \int_0^T (N (u (t), w (t)), u (t))\, dt
\\
 & = &
   \int_0^T (N (u (t)), w (t))\, dt,
\end{eqnarray*}
and the lemma is proved.
\end{proof}
\end{proof}

The solution $u$ given by Theorem \ref{t.exnonNS} satisfies an energy estimate.
In order to show it we integrate the equality
$$
   \partial_t\, \| u_k (t) \|_H^2
 + 2 \left( D (u_k (t)) \! - \! \| u_k (t) \|_H^2 \right)
 = 2\, \Re\, (f (t),u_k (t))
$$
to conclude that
$$
   \| u_k (t) \|_H^2 + 2 \int_0^t \! \left( D (u_k (t')) \! - \! \| u_k (t') \|_H^2 \right) dt'
 = \| u_{0,k} \|_H^2 + 2\, \Re \int_0^t \! (f (t'),u_k (t')) dt'
$$
for all $t \in [0,T]$.
Multiplying this equality by a nonnegative function $\phi \in C^\infty_{\mathrm{comp}} (I)$
and integrating in $t \in I$ yields
\begin{eqnarray*}
\lefteqn{
   \int_0^T
   \Big(
   \| u_k (t) \|_H^2 + 2 \int_0^t \! \left( D (u_k (t')) \! - \! \| u_k (t') \|_H^2 \right) dt'
   \Big)
   \phi (t)
   dt
}
\\
 & = &
   \int_0^T
   \Big( \| u_{0,k} \|_H^2 + 2\, \Re \int_0^t \! (f (t'),u_k (t')) dt' \Big)
   \phi (t)
   dt.
\end{eqnarray*}
Since $u_k$ converges to $u \in L^\infty (I,H) \cap L^2 (I,S)$
   in the weak$^\ast$ topology of $L^\infty (I,H)$,
   weakly in $L^2 (I,S)$ and
   strongly in $L^2 (I,H)$,
we can pass to the lower limit in this relation and obtain
\begin{eqnarray*}
\lefteqn{
   \int_0^T
   \Big(
   \| u (t) \|_H^2 + 2 \int_0^t \! \left( D (u (t')) \! - \! \| u (t') \|_H^2 \right) dt'
   \Big)
   \phi (t)
   dt
}
\\
 & \leq &
   \int_0^T
   \Big( \| u_0 \|_H^2 + 2\, \Re \int_0^t \! (f (t'),u (t')) dt' \Big)
   \phi (t)
   dt
\end{eqnarray*}
for all $\phi \in C^\infty_{\mathrm{comp}} (I)$ satisfying $\phi \geq 0$.
(We have used the fact that if $v_k \to v$ weakly in a normed space $V$ then
   $\| v \|_V \leq \liminf \| v_k \|_V$.)
The last inequality amounts to saying that
\begin{equation}
\label{eq.enerest}
   \| u (t) \|_H^2 + 2 \int_0^t \! \left( V u (t'), u (t') \right) dt'
 \leq
   \| u_0 \|_H^2 + 2\, \Re \int_0^t \! (f (t'),u (t')) dt'
\end{equation}
for almost all $t \in I$.

\section{Regularity and uniqueness}
\label{s.reguniq}

As but one generalisation of Lemma \ref{l.(N(uv)w)} we note that on applying the
H\"{o}lder inequality and the Sobolev embedding theorem one establishes immediately
that the form $(N (u,v),w)$ is actually trilinear continuous on the Cartesian product
   $H^{s_1} (\cX,F^i) \times H^{s_2+1} (\cX,F^i) \times H^{s_3} (\cX,F^i)$,
where
   $s_1 + s_2 + s_3 \geq n/2$, if each of $s_1$, $s_2$ and $s_3$ is different from
   $n/2$,
and
   $s_1 + s_2 + s_3 > n/2$, if some of the $s_1$, $s_2$ and $s_3$ just amounts to
   $n/2$.
That is, there is a constant $c$ depending on $s_1$, $s_2$ and $s_3$, such that
\begin{equation}
\label{eq.tricont}
   |(N (u,v),w)|
 \leq
   c\, \| u \|_{H^{s_1} (\cX,F^i)} \| v \|_{H^{s_2+1} (\cX,F^i)} \| w \|_{H^{s_3} (\cX,F^i)}
\end{equation}
provided $s_1, s_2, s_3 \geq 0$.

Moreover, if $\cX$ is compact, which is the case indeed, then the interpolation inequality
\begin{equation}
\label{eq.interpo}
   \| u \|_{H^{(1-\vartheta) s_1 + \vartheta s_2} (\cX,F^i)}
 \leq
   c\, \| u \|_{H^{s_1} (\cX,F^i)}^{1-\vartheta} \| u \|_{H^{s_2} (\cX,F^i)}^{\vartheta}
\end{equation}
holds for all $u \in H^{s_2} (\cX,F^i)$, where
   $s_1 \leq s_2$,
   $\vartheta \in [0,1]$,
and $c$ is a constant independent of $u$, see
   \cite[Ch.~1]{LionMage72}.

Clearly, various estimates for the trilinear form $(N (u,v),w)$ can be obtained by using
(\ref{eq.tricont}) and (\ref{eq.interpo}).
We mention one of these.

\begin{lemma}
\label{l.totnwst}
Assume that $m \geq (n+2)/4$.
Then there is a constant $c$ with the property that
$$
   |(N (u,v),w)|
 \leq
   c\, \| u \|_H^{1/2} D (u)^{1/4} \| v \|_H^{1/2} D (v)^{1/4} D (w)^{1/2}
$$
for all $u, v, w \in H^m_0 (\cX,F^i)$.
\end{lemma}

\begin{proof}
Choose
$$
\begin{array}{rcl}
       s_1 & = & \displaystyle \frac{1}{2} 0 + \frac{1}{2} m,
\\
   s_2 + 1 & = & \displaystyle \frac{1}{2} 0 + \frac{1}{2} m,
\\
       s_3 & = & m
\end{array}
$$
and apply both (\ref{eq.tricont}) and (\ref{eq.interpo}) with
   $\displaystyle \vartheta = \frac{1}{2}$.
\end{proof}

In particular,
$$
   |(N (u),w)|
 \leq
   c\, \| u \|_H D (u)^{1/2} D (w)^{1/2}
$$
for all $u, w \in S$, and hence
$
   \| N (u) \|_{S'}
 \leq
   c\, \| u \|_H D (u)^{1/2}
$
for all $u \in S$.
If now
   $u \in L^\infty (I,H) \cap L^2 (I,S)$,
then $N (u (t))$ belongs to $S'$ for almost all $t \in I$ and the estimate
$$
   \| N (u (t)) \|_{S'}
 \leq
   c\, \| u (t) \|_H D (u (t))^{1/2}
$$
shows that
   $N (u)$ belongs to $L^2 (I,S')$
and implies
\begin{equation}
\label{eq.efN(u)iS}
   \| N (u) \|_{L^2 (I,S')}
 \leq
   c\, \| u \|_{L^\infty (I,H)} \| u \|_{L^2 (I,S)}
\end{equation}
with $c$ a constant independent of $u$.

We are now in a position to prove the main result of this section, cf.
   Theorem 3.2 of \cite[p.~198]{Tema79}.

\begin{theorem}
\label{t.uniqrNS}
Assume that $m \geq (n+2)/4$.
Then the solution $u$ of problem (\ref{eq.NSweakf}) (or (\ref{eq.NSpowif})) given by
Theorem \ref{t.exnonNS} is unique.
Moreover, $u$ is equal almost everywhere to a function continuous of $t \in [0,T]$
with values in $H$, and
   $u (t) \to u_0$ in $H$, as $t \to 0$.
\end{theorem}

\begin{proof}
\textit{Regularity}.
We first prove the result of regularity.
According to the first equality of (\ref{eq.NSpowif}) we get
$$
   u' = f - Vu - N (u)
$$
almost everywhere on $I$.
By (\ref{eq.efN(u)iS}), each term on the right-hand side of this equality belongs to
$L^2 (I,S')$.
This remark improves the condition $u' \in L^1 (I,S')$, showing that
   $u' \in L^2 (I,S')$.
This improvement enables us to apply Lemma \ref{l.interLM}, which states exactly that
$u$ is almost everywhere equal to a continuous function of $t \in [0,T]$ with values
in $H$.
Thus, $u \in C (\overline{I},H)$, and the last part of Theorem \ref{t.uniqrNS} follows
immediately.

We also recall that Lemma \ref{l.interLM} asserts that for any function $u$ in $L^2 (I,S)$,
such that $u' \in L^2 (I,S')$, the equality
\begin{equation}
\label{eq.dotntp2}
   \partial_t \| u \|_H^2 = 2 \Re\, (u',u)_H
\end{equation}
holds in the sense of distributions on $I$.
This result will next be used in the proof of uniqueness which we now start.

\textit{Uniqueness}.
Suppose $u_1$ and $u_2$ are two solutions of (\ref{eq.NSpowif}).
Consider the difference $u = u_1 - u_2$.
By the above, both $u_1'$ and $u_2'$, and thus $u'$, belong to $L^2 (I,S')$.
The difference satisfies
\begin{equation}
\label{eq.fdu1-u2}
\begin{array}{rclcl}
   u' + Vu
 & =
 & N (u_2) - N (u_1)
 & \mbox{a.e. on}
 & I,
\\
   u (0)
 & =
 & 0.
 &
 &
 \end{array}
\end{equation}
For almost all $t \in I$ we take the scalar product of the first equality in
(\ref{eq.fdu1-u2}) with $u (t)$ in the duality between $S$ and $S'$.
Using (\ref{eq.dotntp2}) yields
\begin{equation}
\label{eq.rospwut}
   \partial_t \| u (t) \|_H^2 + 2 \left( D (u (t)) - \| u (t) \|_H^2 \right)
 = 2 \Re \left( (N (u_2 (t)), u (t)) - (N (u_1 (t)), u (t)) \right).
\end{equation}
From Corollary \ref{c.algestr} it follows that the right-hand side of this equality
just amounts to
$$
   - 2 \Re\, (N (u (t),u_2 (t)), u (t))
   = 2 \Re\, (N (u (t)), u_2 (t)).
$$
On applying Lemma \ref{l.totnwst} we can majorise this expression by
$$
   2 c\, \| u (t) \|_H D (u (t))^{1/2} D (u_2 (t))^{1/2}
 \leq
   \epsilon c^2\, \| u (t) \|_H^2 D (u_2 (t)) + \frac{1}{\epsilon}\, D (u (t)),
$$
where $\epsilon > 0$ is an arbitrary constant.
We choose $\epsilon$ in such a way that $2 - 1/\varepsilon \geq 0$.
Substituting this into (\ref{eq.rospwut}) we conclude that
$$
   \partial_t \| u (t) \|_H^2
 \leq
   \left( 2 + \epsilon c^2\, D (u_2 (t)) \right) \| u (t) \|_H^2
$$
for almost all $t \in I$.
Since the function $t \mapsto D (u_2 (t))$ is integrable, this shows readily that
$$
   \partial_t
   \Big( \exp \Big( - \int_0^t \left( 2 + \epsilon c^2\, D (u_2 (t')) \right) dt' \Big)
         \| u (t) \|_H^2
   \Big)
 \leq
   0.
$$
On integrating and applying the equality $u (0) = 0$ we find $\| u (t) \|_H^2 \leq 0$
for all $t \in [0,T]$.
Hence, $u_1 = u_2$, and the solution is unique, as desired.
\end{proof}

Note that if $n = 2$ then $m = 1$ satisfies the condition $m \geq (n+2)/4$.
Hence, we recover the classical result on the existence and uniqueness of a weak solution
to the Navier-Stokes equations in dimension $2$,
   see for instance Theorem 3.2 of \cite[p.~198]{Tema79}.

It is also worth pointing out that
$$
   L^\infty (I,H) \cap L^2 (I,S)
 \hookrightarrow
   L^{\frac{\scriptstyle 2}{\scriptstyle 1 - \vartheta}} (I, H^{(1-\vartheta) m} (\cX,F^i))
$$
for all $\vartheta \in [0,1]$,
   see \cite{LionMage72} for an interpolation argument.
The reader may also consult \cite{BatzBere67}.

\part{Limit to the Navier-Stokes equations}
\label{s.lttNSeq}

\section{More regular solutions}
\label{s.moresol}

Our purpose in this section is to prove that on assuming more regularity of the data $f$ and
$u_0$ we can obtain more regular solutions to the regularised Navier-Stokes equations,
   cf. Theorem 3.5 of \cite[p.~202]{Tema79}.

The domain $\cD_V$ of the closed unbounded operator $V$ in $H$ is well known to consist of
all sections $u \in H^{2m} (\cX;F^i)$ which satisfy $A^\ast u = 0$ in $\cX$ and vanish up to
order $m$ on the boundary of $\cX$.

\begin{lemma}
\label{l.regtfNS}
Suppose $m \geq (n+2)/4$.
If
   $u \in \cD_V$
then
   $N (u) \in H$
and
$$
   \| N (u) \|_H \leq c\, \| u \|_H^{1/2} D (u)^{1/2}\, \| (V + \Id) u \|_H^{1/2}
$$
with $c$ a constant independent of $u$.
\end{lemma}

\begin{proof}
We exploit the estimate of (\ref{eq.tricont}) with
   $\displaystyle s_1 = \frac{1}{2} m$,
   $\displaystyle s_2 + 1 = \frac{1}{2} m + \frac{1}{2} 2m$ and
   $s_3 = 0$,
obtaining
$$
   |(N (u,v),w)|
 \leq
   c\, \| u \|_{H^{s_1} (\cX,F^i)} \| v \|_{H^{s_2+1} (\cX,F^i)} \| w \|_{H^{s_3} (\cX,F^i)}
$$
for all
   $u \in H^{s_1} (\cX,F^i)$,
   $v \in H^{s_2+1} (\cX,F^i)$ and
   $w \in H^{s_3} (\cX,F^i)$,
where $c$ is a constant independent of $u$, $v$ and $w$.
By the interpolation inequality of (\ref{eq.interpo}), the norm of $u$ in $H^{s_1} (\cX,F^i)$
is dominated by $\| u \|_H^{1/2} D (u)^{1/4}$.
If moreover $v$ belongs to $H^{2m} (\cX,F^i)$, then on applying inequality (\ref{eq.interpo})
with
   $s_1 = m$,
   $s_2 = 2m$ and
   $\vartheta = 1/2$,
we deduce that the norm of $v$ in $H^{s_2+1} (\cX,F^i)$ is majorised by
   $D (v)^{1/4} \| v \|_{H^{2m} (\cX,F^i)}^{1/2}$.
Since the Dirichlet problem for the operator $V + \Id$ in $\cX$ is elliptic and has unique
solution, the norm $\| v \|_{H^{2m} (\cX,F^i)}$ is in turn majorised by $\| (V+\Id) v \|_H$,
provided that $v \in \cD_V$.
On summarising we see immediately that there is a constant $c$ with the property that
\begin{equation}
\label{eq.peoNuvw}
   |(N (u,v),w)|
 \leq
   c\, \| u \|_H^{1/2} D (u)^{1/4} D (v)^{1/4} \| (V+\Id) v \|_H^{1/2} \| w \|_H
\end{equation}
for all
   $u \in S$,
   $v \in \cD_V$ and
   $w \in H$.
By the Riesz representation theorem it follows readily that
   $\| N (u,u) \|_H \leq c\, \| u \|_H^{1/2} D (u)^{1/2} \| (V+\Id) u \|_H^{1/2}$
for all sections $u \in \cD_V$, as desired.
\end{proof}

The following lemma is usually referred to as Gronwall's inequality, see
   \cite{Gron19}.
It is an important tool to obtain various estimates in the theory of ordinary
differential equations.
In particular, it provides a comparison theorem that can be used to prove the
uniqueness of a solution to the initial value problem.
There are two forms of the lemma, a differential form and an integral form.
We adduce here the integral form.

\begin{lemma}
\label{l.Gronwall}
Let $a$, $b$ and $y$ be real-valued functions on an interval $I = (0,T)$.
Assume that
   $a$ and $y$ are continuous on the closed interval $\overline{I}$
and
   the negative part of $b$ is locally integrable in $I$.
If $a$ is nonnegative and if $y$ satisfies the integral inequality
$$
   y (t) \leq b (t) + \int_0^t a (t') y (t') dt'
$$
for all $t \in I$, then
$$
   y (t) \leq b (t) + \int_0^t b (s) a (s) \exp \Big( \int_s^t a (t') dt' \Big) ds
$$
for all $t \in I$.
\end{lemma}

If, in addition, the function $b$ is nondecreasing, then
\begin{equation}
\label{eq.Gronwall}
   y (t) \leq b (t)\, \exp \Big( \int_0^t a (t') dt' \Big)
\end{equation}
for all $t \in I$.

\begin{proof}
See \cite{Bell43}.
\end{proof}

It is worth pointing out that there are no assumptions on the signs of the functions
$b$ and $y$.
Compared to the differential form, the differentiability of $y$ is not needed for the
integral form.
There are also versions of Gronwall's inequality which need not any continuity of $a$
and $y$.

The lemma is very useful when establishing a uniform estimate in a priori esimates.

\begin{theorem}
\label{t.regtfNS}
Suppose $m \geq (n+2)/4$.
Let
   $f \in L^2 (I,H)$ and
   $u_0 \in S$.
Then there is a unique solution $u$ of problem (\ref{eq.NSweakf}) (or (\ref{eq.NSpowif}))
which satisfies
   $u \in L^2 (I,\cD_V)$ and
   $u' \in L^2 (I,H)$,
i.e.,
\begin{equation}
\label{eq.domrNSe}
\begin{array}{rcl}
   u & \in & H^{\bs (1)} (\cC,F^{i}),
\\
   p & \in & L^2 (I,H^1 (\cX,F^{i-1})).
\end{array}
\end{equation}
\end{theorem}

Moreover, the proof shows that the solution $u$ given by this theorem belongs to
   $C (\overline{I},S)$.

\section{Proof of Theorem \ref{t.regtfNS}}
\label{s.proofTh}

\begin{proof}
Consider the Galerkin approximation used in the proof of weak solutions in Theorem
\ref{t.exnonNS}.
In order to obtain further regularity properties of solutions, we choose the basis
functions $e_i$ as the eigenfunctions of the operator $V+\Id$ in (\ref{eq.NSpowif}).
We get $e_i \in \cD_V$ and $(V+\Id) e_i = \lambda_i e_i$ for all $i = 1, 2, \ldots$,
where $\lambda_i \geq 1$, for $V$ is nonnegative.
Write $u_{0,k}$ for the orthogonal projection in $S$ of $u_0$ onto the space spanned by
the elements $e_1, \ldots, e_k$.
Hence it follows that $u_{0,k}$ converges to $u_0$ strongly in $S$ as $k \to \infty$.

The first equations of (\ref{eq.appsNSc}) read
$$
   (u_k',e_j) + (V u_k,e_j) + (N (u_k),e_j) = (f,e_j)
$$
for all $j = 1, \ldots, k$.
On multiplying both sides of these equalities by $\lambda_j$ we obtain immediately
$$
   (u_k',(V+\Id) e_j) + (V u_k,(V+\Id) e_j) + (N (u_k),(V+\Id) e_j) = (f,(V+\Id) e_j),
$$
and so multiplying the latter equalities by $c_{k,j} (t)$ and summing up over
$j = 1, \ldots, k$ yields
\begin{equation}
\label{eq.uoj1lky}
   (u_k',(V+\Id) u_k) + (V u_k,(V+\Id) u_k) + (N (u_k),(V+\Id) u_k) = (f,(V+\Id) u_k)
\end{equation}
a.e. on $I$.
Since
$$
   \Re\, (u_k',(V+\Id) u_k) = \frac{1}{2}\, \partial_t D (u_k)
$$
and
\begin{eqnarray*}
   |(N (u_k),(V+\Id) u_k)|
 & \leq &
   \| N (u_k) \|_H \| (V+\Id) u_k \|_H
\\
 & \leq &
   c\,
   \| u_k \|_H^{1/2}
   D (u_k)^{1/2}\,
   \| (V + \Id) u_k \|_H^{3/2}
\end{eqnarray*}
which is due to Lemma \ref{l.regtfNS},
Therefore, from (\ref{eq.uoj1lky}) we conclude by the Schwarz inequality that
\begin{eqnarray*}
\lefteqn{
   \frac{1}{2}\, \partial_t D (u_k) + \| (V+\Id) u_k \|_H^2
}
\\
 & \leq &
   \big( \| f \|_H + \| u_k \|_H \big) \| (V+\Id) u_k \|_H
 + c\, \| u_k \|_H^{1/2} D (u_k)^{1/2}\, \| (V + \Id) u_k \|_H^{3/2}
\end{eqnarray*}
a.e. on $I$.

By Young's inequality,
\begin{eqnarray*}
\lefteqn{
   2 \big( \| f \|_H + \| u_k \|_H \big) \| (V+\Id) u_k \|_H
 + 2c\, \| u_k \|_H^{1/2} D (u_k)^{1/2}\, \| (V + \Id) u_k \|_H^{3/2}
}
\\
 & \leq &
   a_k (t) D (u_k) + b_k (t) + \| (V+\Id) u_k \|_H^2,
   \hspace{4cm}
\end{eqnarray*}
where
$$
\begin{array}{rcl}
   a_k (t)
 & = &
   \displaystyle \frac{27}{2} c^4 \| u_k \|_H^2 D (u_k),
\\
   b_k (t)
 & = &
   2 \big( \| f \|_H + \| u_k \|_H \big)^2.
\end{array}
$$
We thus arrive at the differential inequality
\begin{equation}
\label{eq.difineq}
   \partial_t D (u_k) + \| (V+\Id) u_k \|_H^2 \leq a_k (t) D (u_k) + b_k (t)
\end{equation}
for almost all $t \in I$.
Since $u_k$ are bounded uniformly in $k$ both in $L^\infty (I,H)$ and $L^2 (I,S)$,
it follows that
$$
   \int_t^{t+r} a_k (t') dt'
 \leq
   \frac{27}{2} c^4\,
   \sup_{t' \in I} \| u_k (t') \|_H^2
   \int_{t}^{t+r} D (u_k (t')) dt'
 \leq
   A
$$
for all $t \in I$ satisfying $t+r \leq T$, the constant $A$ being independent of $k$.
Similarly we obtain
$$
   \int_t^{t+r} b_k (t') dt'
 \leq
   4 \int_t^{t+r} \big( \| f (t') \|_H^2 + \| u_k (t') \|_H^2 \big) dt'
 \leq
   B
$$
for all $k$, with $B$ a constant independent of $k$.
We now drop the term $\| (V+\Id) u_k \|_H^2$ in (\ref{eq.difineq}) and apply the uniform
Gronwall inequality of Lemma \ref{l.Gronwall} to the resulting inequality to get
\begin{equation}
\label{eq.uniesiS}
   \sup_{t \in [0,T]} D (u_k (t)) \leq C
\end{equation}
uniformly in $k$.
This just amounts to saying that the sequence $\{ u_k \}$ belongs to a bounded set in
$L^\infty (I,S)$.

Now go back to inequality (\ref{eq.difineq}) and integrate it in $t$ over the interval $I$.
This gives
$$
   D (u_k (T)) + \int_0^T \| (V+\Id) u_k (t) \|_H^2 dt
 \leq
   D (u_0) + \int_0^T a_k (t) D (u_k (t)) dt + \int_0^T b_k (t) dt
$$
for all $k = 1, 2, \ldots$.
Hence,
\begin{equation}
\label{eq.uniesiD}
   \int_0^T \| (V+\Id) u_k (t) \|_H^2 dt
 \leq
   D (u_0) + A\, C + B
\end{equation}
holds uniformly in $k$, where $C$ is a constant from (\ref{eq.uniesiS}).
We have thus proved that $\{ u_k \}$ belongs to a bounded set in $L^2 (I,\cD_V)$.

The passage to the limit and the uniqueness is established as above.
We can conclude that some subsequence of $\{ u_k \}$ converges to the solution $u$
which belongs to $L^\infty (I,S) \cap L^2 (I,\cD_V)$.

Our next goal is to show that $u' \in L^2 (I,H)$.
For this purpose we use Lemma \ref{l.regtfNS} to estimate
\begin{eqnarray*}
   \int_0^T \| N (u) \|_H^4 dt
 & \leq &
   c \int_0^T \| u \|_H^2 D (u)^2\, \| (V + \Id) u \|_H^2 dt
\\
 & \leq &
   c\, \| u \|_{L^\infty (I,H)}^2 \| u \|_{L^\infty (I,S)}^4 \| u \|_{L^2 (I,\cD_V)}^2,
\end{eqnarray*}
where $c$ is a constant independent of $u$.
This implies $N (u) \in L^4 (I,H)$.
On the other hand, since $V + \Id$ is an isomorphism of $\cD_V$ onto $H$, it follows
easily that $Vu \in L^2 (I,H)$.
Therefore, the derivative $u' = f - Vu - N (u)$ belongs to $L^2 (I,H)$, as desired.

Now
   $u \in L^\infty (I,S) \cap L^2 (I,\cD_V)$ and
   $u' \in L^2 (I,H)$
imply $u \in C (\overline{I},S)$ due to the lemma below.
This establishes the strong convergence $u (t) \to u (0)$ in $S$, as $t \to 0$.
\end{proof}

\begin{lemma}
\label{l.tetscut}
Suppose that
   $u \in L^\infty (I,S) \cap L^2 (I,\cD_V)$ and
   $u' \in L^2 (I,H)$.
Then $u \in C (\overline{I},S)$.
\end{lemma}

\begin{proof}
Fix an arbitrary point $t_0 \in [0,T]$.
We have to show that $D (u (t) - u (t_0)) \to 0$, as $t \to t_0$.
That is,
$$
   \lim_{t \to t_0}
   \big( D (u (t)) - 2 \Re\, D (u (t),u (t_0)) + D (u (t_0)) \Big)
 = 0.
$$

A direct calculation shows that $D (u (t),u (t_0))$ reduces to
$$
   \varepsilon\, (\iD^{m/2} u (t),\iD^{m/2} u (t_0))
 + \nu \big( (A u (t),A u (t_0)) + (A^\ast u (t),A^\ast u (t_0)) + (u (t),u (t_0)) \big).
$$
By Lemma \ref{l.interLM}, the function $u$ is equal a.e. to a continuous function on
$I$ with values in $H$, for
   $u \in L^2 (I,S)$ and
   $u' \in L^2 (I,S')$.
Since moreover
   $u \in L^\infty (I,S)$ and
   the embedding $S \hookrightarrow H$ is continuous and has dense range,
it follows that $u$ is weakly continuous from $[0,T]$ into $S$, i.e.,
   the function $t \mapsto D (u (t),v)$ is continuous for all $v \in S$,
see Lemma 1.4 in \cite[p.~178]{Tema79}.
In particular,
   $D (u (t),u (t_0)) \to D (u (t_0))$,
as $t \to t_0$.

We will now show that $D (u (t)) \to D (u (t_0))$, as $t \to t_0$.
To this end, we first establish that
$$
   \partial_t\, D (u (t)) = 2 \Re\, (u', (V + \Id) u)
$$
holds in the sense of distributions on $(0,T)$.
Note that both
   $t \mapsto D (u (t))$ and
   $t \mapsto (u' (t),(V + \Id) u (t))$
are integrable functions on the interval $I$, for
   $u \in L^2 (I,\cD_V)$ and
   $u' \in L^2 (I,H)$.

Choose any sequence of smooth functions $u_\iota \in C^\infty (\R,\cD_V)$ which is
obtained by mollifying the function which is equal to $u$ on the interval $[0,T]$ and
to zero outside this interval.
(Mollifiers are also known as approximations to the identity.)
Since $u \in L^2 (I,\cD_V)$, the sequence $\{ u_\iota \}$ converges to $u$ in the norm
of this space.
Similarly, as $u' \in L^2 (I,H)$, the sequence of derivatives $\{ u_\iota' \}$ converges
to $u'$ in the $L^2 (I,H)\,$-norm.
This implies that
$$
\begin{array}{rcl}
   D (u_\iota (t))
 & \to
 & D (u (t)),
\\
   (u_\iota' (t), (V + \Id) u_\iota (t))
 & \to
 & (u' (t), (V + \Id) u (t))
\end{array}
$$
in $L^1 (I)$, as $\iota \to \infty$.
To clarify the second limit passage, we argue in the following way.
We get
\begin{eqnarray*}
\lefteqn{
   \int_0^T |(u', (V + \Id) u) - (u_\iota', (V + \Id) u_\iota)|\, dt
}
\\
 & \! \leq \! &
   \int_0^T
   \big( \| u' - u_\iota' \|_H \| (V + \Id) u \|_H
       + \| u_\iota' \|_H \| (V + \Id) u - (V + \Id) u_\iota \|_H
   \big)
   dt
\\
 & \! \leq \! &
   \| u'\!-\!u_\iota' \|_{L^2 (I,H)} \| (V\!+\!\Id) u \|_{L^2 (I,H)}
 + \| u_\iota' \|_{L^2 (I,H)} \| (V\!+\!\Id) u - (V\!+\!\Id) u_\iota \|_{L^2 (I,H)}
\end{eqnarray*}
for all $\iota$.
Since
   $u \in L^2 (I,\cD_V)$,
   the norms $\| u_\iota' \|_{L^2 (I,H)}$ are bounded uniformly in $\iota$ and
   $(V + \Id) u_\iota$ converges to $(V + \Id) u$ in the $L^2 (I,H)\,$-norm,
we see immediately that
   $(u_\iota', (V + \Id) u_\iota) \to (u', (V + \Id) u)$
in $L^1 (I)$ and, in particular, in the sense of distributions on $I$,
   as $\iota \to \infty$.

Note that from $D (u_\iota (t)) \to D (u(t))$ in the sense of distributions on $I$ it
follows that
$$
   \partial_t D (u_\iota (t)) \to \partial_t D (u (t))
$$
in the same sense on $I$.
Thus,
$
   \partial_t\, D (u (t)) = 2 \Re\, (u' (t), (V + \Id) u (t))
$
is valid in the sense of distributions on $(0,T)$.
Since
$
   (u' (t), (V + \Id) u (t))
$
is integrable on $I$, we write
$$
   D (u (t)) - D (u (t_0)) = 2 \Re \int_{t_0}^t (u' (t'), (V + \Id) u (t')) dt'
$$
and conclude that $D (u (t)) \to D  (u (t_0))$, as $t \to t_0$, showing the lemma.
\end{proof}

\section{Estimates for the solution uniform in parameter}
\label{s.eftsuip}

When comparing to the conventional Navier-Stokes equations, the regularised equations
of (\ref{eq.NSpowif}) contain a higher order viscosity term $\iD^m u$ multiplied by a
small parameter $\varepsilon > 0$.
Theorem \ref{t.exnonNS} asserts that, given any data
   $f \in L^2 (I,S')$ and
   $u_0 \in H$,
problem (\ref{eq.NSpowif}) possesses at least one solution
   $u \in L^2 (I,S) \cap L^\infty (I,H)$
satisfying
   $u' \in L^1 (I,S')$.
Moreover, if $m$ is large enough, to wit $m \geq (n+2)/4$, then this solution is unique.
From now on we write $u_{\varepsilon}$ for this solution to point out its dependence upon
$\varepsilon$.
Our next concern will be the behaviour of $u_{\varepsilon}$ when $\varepsilon \to 0$.
One might expect naively that the family $u_{\varepsilon}$ converges to a weak solution
of the conventional equations.
However, the problem under study is specified within the framework of singularly perturbed
boundary value problems, let alone nonlinear, see \cite{GlebKiseTark18}.
If $\varepsilon = 0$, then the order of the regularised Navier-Stokes equations in $x$
reduces to $2$.
This corresponds to $m = 1$, in which case the condition $m \geq (n+2)/4$ is fulfilled
only for small dimensions $n = 1$ and $n = 2$.
While the existence of a weak solution $u \in L^2 (\cC,F^i)$ is still known for all $n > 2$,
see \cite{Hopf51}, there have been indirect plausible arguments showing that the solution
fails to be unique.
Hence, there seems to be no natural weak solution of the conventional Navier-Stokes equations
to which the family $u_{\varepsilon}$ might converge.
Perhaps the family has a number of accumulation points, each of them being a weak solution
to the classical Navier-Stokes equations in certain sense.
Whatever the case one needs a priori estimates for the solution $u_{\varepsilon}$ uniform
in $\varepsilon$ to establish the existence of accumulation points in a weak topology.

\begin{lemma}
\label{l.aeuniie}
Let $m \geq (n+2)/4$.
If $f \in L^2 (I,H)$, then the weak solution of Theorem \ref{t.exnonNS} satisfies
\begin{eqnarray*}
\lefteqn{
   \| u_{\varepsilon} (t) \|_H^2
 + 2 \nu
   \int_0^t
   e^{(1/c) (t-t')}
   \big( \| A u_{\varepsilon} (t') \|_H^2 + \| A^\ast u_{\varepsilon} (t') \|_H^2 \big)
   dt'
}
\\
 & \leq &
   e^{(1/c) t}\, \| u_0 \|_H^2
 + c
   \int_0^t
   e^{(1/c) (t-t')}\, \| f (t') \|_H^2 dt'
   \hspace{1cm}
\end{eqnarray*}
for almost all $t \in I$, where
   $c$ is an arbitrary positive constant independent of $t$ and $\varepsilon$.
\end{lemma}

\begin{proof}
Since $u_{\varepsilon} \in L^2 (I,S)$, we can multiply both sides of (\ref{eq.NSpowif})
by this section, obtaining
$$
   (u_\varepsilon' (t),u_\varepsilon (t))
 + (V u_\varepsilon (t),u_\varepsilon (t))
 + (N (u_\varepsilon (t)),u_\varepsilon (t))
 = (f (t),u_\varepsilon (t))
$$
for almost all $t \in I$.
Using the equality
$
   \partial_t \| u_\varepsilon (t) \|_H^2 = 2 \Re\, (u_\varepsilon' (t),u_\varepsilon (t)),
$
which is due to (\ref{eq.efN(u)iS}), and
$
   (N (u_\varepsilon (t)),u_\varepsilon (t)) = 0
$
we get
\begin{eqnarray*}
   \partial_t \| u_\varepsilon (t) \|_H^2
 + 2\, (V u_\varepsilon (t),u_\varepsilon (t))
 & = &
   2 \Re\, (f (t),u_\varepsilon (t))
\\
 & \leq &
   2\, \| f (t) \|_H \| u_\varepsilon (t) \|_H
\\
 & \leq &
   c\, \| f (t) \|_H^2 + \frac{1}{c}\,  \| u_\varepsilon (t) \|_H^2
\end{eqnarray*}
a.e. on $I$, where $c$ is an arbitrary positive constant.
On the other hand, one easily checks that
$$
   \partial_t \| u_\varepsilon (t) \|_H^2 - \frac{1}{c}\,  \| u_\varepsilon (t) \|_H^2
 = e^{(1/c) t} \partial_t \big( e^{- (1/c) t} \| u_\varepsilon (t) \|_H^2 \big)
$$
whence
$$
   \partial_t \big( e^{- (1/c) t} \| u_\varepsilon (t) \|_H^2 \big)
 + 2 e^{- (1/c) t}\, (V u_\varepsilon (t),u_\varepsilon (t))
 \leq
   c e^{- (1/c) t}\, \| f (t) \|_H^2
$$
a.e. on the interval $I$.
Integrating this inequality from $0$ to $t$, where $t$ is a fixed point of $I$, we
obtain
\begin{eqnarray*}
\lefteqn{
   \| u_\varepsilon (t) \|_H^2
 + 2 \int_0^t e^{(1/c) (t-t')}\, (V u_\varepsilon (t'),u_\varepsilon (t')) dt'
}
\\
 & \leq &
   e^{(1/c) t}\, \| u_0 \|_H^2
 + c\, \int_0^t e^{(1/c) (t-t')}\, \| f (t') \|_H^2 dt'.
   \hspace{1cm}
\end{eqnarray*}
As
$$
\begin{array}{rcl}
   (V u_\varepsilon (t'),u_\varepsilon (t'))
 & \geq &
   (\nu\, \iD u_\varepsilon (t'),u_\varepsilon (t'))
\\
 & \geq &
    \nu \big( \| A u_{\varepsilon} (t') \|_H^2 + \| A^\ast u_{\varepsilon} (t') \|_H^2 \big),
\end{array}
$$
the lemma follows.
\end{proof}

Perhaps the choice $c = 1$ is optimal.

\section{A uniqueness theorem for the conventional equations}
\label{s.uniqthe}

We now consider the unperturbed equations corresponding to $\varepsilon = 0$, which are
the classical Navier-Stokes equations.
Recall that $S^1$ stands for the closure of $\mathcal{S}$ in $H^1 (\cX,F^i)$.
Given
   $f \in L^2 (I, S^1{}')$ and
   $u_0 \in H$,
we look for $u \in L^\infty (I,H) \cap L^2 (I,S^1)$ satisfying
\begin{equation}
\label{eq.conveNS}
\begin{array}{rclcl}
   (u',v) + \nu (\iD u,v) + (N (u),v)
 & =
 & (f,v)
 & \mbox{for all}
 & v \in S^1 \cap L^n (\cX,F^i),
\\
   u (0)
 & =
 & u_0
 & \mbox{on}
 & \cX.
 \end{array}
\end{equation}

Note that the condition $u (0) = u_0$ is interpreted in the same way as in problem
(\ref{eq.NSweakf}).
The additional condition $v \in S^1 \cap L^n (\cX,F^i)$ is used to make the trilinear
form $(N (u,v),w)$ continuous on $S^1 \times S^1 \times (S^1 \cap L^n (\cX,F^i))$,
   cf. Lemma 6.1 of \cite[p.~79]{Lion69}.

By the above, there is at least one solution $u \in L^\infty (I,H) \cap L^2 (I,S^1)$
to problem (\ref{eq.conveNS}).
The question if this solution is unique remains still open.
One may ask what additional conditions on $u$ guarantee its uniqueness.
The following theorem is due to \cite{Lady70}.

\begin{theorem}
\label{t.uniqcla}
The solution of (\ref{eq.conveNS}) in
   $L^r (I,L^q (\cX,F^i))$,
if exists, is unique, provided that $q > n$ and
$$
   \frac{2}{r} + \frac{n}{q} \leq 1.
$$
\end{theorem}

\begin{proof}
We consider the most interesting case where
   $2/r + n/q = 1$.
Using the H\"{o}lder inequality, we get
\begin{equation}
\label{eq.utHiqwg}
   | (N (u,v),w) |
 \leq
   c\, \| u \|_{L^q (\cX,F^i)} \| v \|_{H^1 (\cX,F^i)} \| w \|_{L^s (\cX,F^i)}
\end{equation}
where $1/q + 1/s = 1/2$.
On the other hand, if $\phi$ is a function with scalar values, then
$$
   \| \phi \|_{L^s (\cX)}
 \leq
   \| \phi \|_{L^2 (\cX)}^{2/r} \| \phi \|_{L^{2n/(n-2)} (\cX)}^{n/q},
$$
for
$$
   \frac{1}{s} = \frac{2/r}{2} + \frac{n/q}{2n/(n-2)}.
$$
We have used here the inequality
$
   \| \phi \|_{L^s (\cX)}
 \leq
   \| \phi \|_{L^{s_1 p} (\cX)}^{s_1/s} \| \phi \|_{L^{s_2 p'} (\cX)}^{s_2/s},
$
where $p'$ is the dual exponent for $p$ and $s_1 + s_2 = s$.
Since $H^1_0 (\cX,F^i) \hookrightarrow L^{2n/(n-1)} (\cX,F^i)$, it follows that
$$
   \| w \|_{L^s (\cX,F^i)}
 \leq
   c\, \| w \|_{L^2 (\cX,F^i)}^{2/r} \| w \|_{H^1 (\cX,F^i)}^{n/q}
$$
with $c$ a constant independent of $w$ and different in diverse applications.
Hence, (\ref{eq.utHiqwg}) implies that
\begin{equation}
\label{eq.utHiwgc}
   | (N (u,v),w) |
 \leq
   c\,
   \| u \|_{L^q (\cX,F^i)}
   \| v \|_{H^1 (\cX,F^i)}
   \| w \|_{H}^{2/r}
   \| w \|_{H^1 (\cX,F^i)}^{n/q}.
\end{equation}

Our next objective is to show that if $u$ is a solution of problem (\ref{eq.NSweakf}),
satisfying $u \in L^r (I,L^q (\cX,F^i))$ with $q > n$ and $2/r + n/q \leq 1$, then
   $u' \in L^2 (I,S^1{}')$.
Indeed, from (\ref{eq.utHiwgc}) it follows that
\begin{eqnarray*}
   | (N (u),v) |
 & = &
   | - (N (u,v),u) |
\\
 & \leq &
   c\,
   \| u \|_{L^q (\cX,F^i)}
   \| v \|_{H^1 (\cX,F^i)}
   \| u \|_{H}^{2/r}
   \| u \|_{H^1 (\cX,F^i)}^{n/q}
\\
 & \leq &
   c\,
   \| u \|_{L^q (\cX,F^i)}
   \| v \|_{H^1 (\cX,F^i)}
   \| u \|_{H^1 (\cX,F^i)}^{n/q},
\end{eqnarray*}
for the norm $\| u \|_H$ is bounded uniformly in $u$.
Since
   the function $t \mapsto \| u (t) \|_{L^q (\cX,F^i)}$ belongs to $L^r (I)$ and
   the function $t \mapsto \| u (t) \|_{H^1 (\cX,F^i)}^{n/q}$ belongs to $L^{2q/n} (I)$,
we see that their product
$$
   t \mapsto \| u (t) \|_{L^q (\cX,F^i)} \| u (t) \|_{H^1 (\cX,F^i)}^{n/q}
$$
belongs to $L^2 (I)$.
By the Riesz representation theorem, the section $t \in N (u (t))$ is of class
$L^2 (I,S^1{}')$.
Hence it follows that $u' = f - \nu \iD u - N (u)$ belongs to $L^2 (I,S^1{}')$,
as desired.

In particular, on applying Lemma \ref{l.interLM} we conclude that the mapping
$u : I \to H$ is continuous.

Having disposed of this preliminary step, we can now establish the uniqueness of the
solution $u$.
Suppose that $u_1$ and $u_2$ be two solutions of problem (\ref{eq.NSweakf}) of class
   $L^r (I,L^q (\cX,F^i))$,
where
   $q > n$ and
   $2/r + n/q \leq 1$.
Write $u = u_1 - u_2$.
Then we obtain
$$
   (u',v) + \nu\, (\iD u,v) + (N (u,u_1),v) + (N (u_1,u),v) - (N (u),v) = 0
$$
for all $v \in S^1 \cap L^n (\cX,F^i)$.
On substituting $u$ for the test section $v$ (which is allowed, for $u' \in L^2 (I,S^1{}')$)
we get
\begin{eqnarray*}
   \frac{1}{2} \partial_t \| u (t) \|_H^2 + \nu\, (\iD u (t),u (t))
 & = &
   -\, \Re\, (N (u (t),u_1 (t)),u (t))
\\
 & = &
   \Re\, (N (u (t)),u_1 (t))
\end{eqnarray*}
for almost all $t \in I$.

Using a version of inequality (\ref{eq.utHiwgc}) yields
$$
   | (N (u (t)),u_1 (t)) |
 \leq
   c\,
   \| u_1 (t) \|_{L^q (\cX,F^i)}
   \| u (t) \|_{H}^{2/r}
   \| u (t) \|_{H^1 (\cX,F^i)}^{n/q + 1}
$$
where $c$ is a constant independent of both $u_1$, $u_2$ and $t$.
Set
$$
   a (t) = \| u_1 (t) \|_{L^q (\cX,F^i)}^r,
$$
then $a \in L^1 (I)$ and
\begin{eqnarray*}
   | (N (u (t)),u_1 (t)) |
 & \leq &
   c\,
   (a (t))^{1/r}
   \| u (t) \|_{H}^{2/r}
   \| u (t) \|_{H^1 (\cX,F^i)}^{n/q + 1}
\\
 & \leq &
   \Big( c C (a (t))^{1/r} \| u (t) \|_{H}^{2/r} \Big)
   \Big( \frac{1}{C}\, \| u (t) \|_{H^1 (\cX,F^i)}^{n/q + 1} \Big),
\end{eqnarray*}
where $C$ is an arbitrary positive constant to be determined.
By Young's inequality for the product we obtain
$$
   | (N (u (t)),u_1 (t)) |
 \leq
   \frac{(c C)^r}{r}\, a (t) \| u (t) \|_{H}^2
 + \frac{1}{r' C^{r'}}\, \| u (t) \|_{H^1 (\cX,F^i)}^2
$$
where $r' = r/(r-1)$ is the dual exponent for $r$.
We have used here the inequality $2/r + n/q \leq 1$.

On summarising we deduce that
\begin{eqnarray*}
\lefteqn{
   \frac{1}{2} \partial_t \| u (t) \|_H^2
 + \nu\, ((\iD + \Id) u (t),u (t))
 - \frac{1}{r' C^{r'}}\, \| u (t) \|_{H^1 (\cX,F^i)}^2
}
\\
 & \leq &
   \Big( \frac{(c C)^r}{r}\, a (t) + \nu \Big) \| u (t) \|_H^2
   \hspace{3cm}
\end{eqnarray*}
for almost all $t \in I$.
Since the Dirichlet problem for the differential operator $\iD + \Id$ is elliptic,
there is a sufficiently large constant $C > 0$ independent of $u$ and $t \in I$,
such that
$$
   \nu\, ((\iD + \Id) u (t),u (t))
 \geq
   \frac{1}{r' C^{r'}}\, \| u (t) \|_{H^1 (\cX,F^i)}^2.
$$
It follows that
$$
   \frac{1}{2} \partial_t \| u (t) \|_H^2
 \leq
   \Big( \frac{(c C)^r}{r}\, a (t) + \nu \Big) \| u (t) \|_H^2
$$
for alsmost all $t \in I$.
Since $u (0) = 0$, on arguing as in the proof of Theorem \ref{t.uniqrNS} we see that
$u \equiv 0$, as desired.
\end{proof}

\section{Limit to the conventional Navier-Stokes equations}
\label{s.lttcNSe}

In this section we establish that the weak solution
   $u_\varepsilon \in C (\overline{I},H) \cap L^2 (I,S)$
to the regularised Navier-Stokes equations has a strong limit to a weak solution to the
Navier-Stokes equations.
We will consider a limit passage for the weak solution of problem (\ref{eq.NSweakf}) whose
existence and uniqueness have been proved in Theorems \ref{t.exnonNS} and \ref{t.uniqrNS}.

By Theorem \ref{t.regtfNS}, if
   $f \in L^2 (I,H)$ and
   $u_0 \in S$,
then the unique solution $u_\varepsilon$ of problem (\ref{eq.NSweakf}) belongs to the space
   $u_\varepsilon \in L^2 (I,\cD_V)$.
Since $m \geq (n+2)/4$ is greater than $n/4$, we conclude by the Sobolev embedding theorem
that $\cD_V \hookrightarrow C (\cX,F^i)$.
Therefore, Theorem \ref{t.uniqcla} implies that the family $u_\varepsilon$ belongs to
a uniqueness class for problem (\ref{eq.conveNS}).
However, the family need not be bounded in this class uniformly in $\varepsilon$.

On the other hand, Lemma \ref{l.aeuniie} shows that the family $u_\varepsilon$ is bounded
uniformly in $\varepsilon$ both in
   $L^\infty (I,L^2 (\cX,F^i))$ and
   $L^2 (I,H^1 (\cX,F^i))$.
By Theorem \ref{t.uniqcla}, the former space is a uniqueness class for the conventional
problem, if $n \leq 2$, while the second one is such merely for $n < 2$.
More generally, on combining Theorem \ref{t.uniqcla} with the Sobolev embedding theorem we
see that $L^2 (I,H^m (\cX,F^i))$ is a uniqueness class for problem (\ref{eq.conveNS}),
provided that $m > n/2$.
Although the latter condition is stronger than $m \geq (n+2)/4$, the authors find the
boundedness of the family $u_\varepsilon$ in $L^2 (I,H^m (\cX,F^i))$ an optimal condition
for the singular perturbation under consideration.

\begin{theorem}
\label{t.lttcNSe}
Suppose $m > n/2$.
Let $u_\varepsilon$ be the weak solution of problem (\ref{eq.NSweakf}) given by
Theorem \ref{t.regtfNS} with 
   $f \in L^2 (I,H)$ and
   $u_0 \in S$.
If the family $u_\varepsilon$ is bounded uniformly in $\varepsilon$ in $L^2 (I,H^m (\cX,F^i))$,
then it approaches the unique weak solution $u$ to the conventional problem of (\ref{eq.conveNS})
   in the norm of $L^\infty (I,H) \cap L^2 (I,S^1)$ and
   weakly in $L^2 (I, S)$.
\end{theorem}

\begin{proof}
We first estimate the difference $u_\varepsilon - u_\delta$ for any positive $\delta$ and
$\varepsilon$.
To be specific, consider $0 < \delta < \varepsilon$.

By Theorem \ref{t.regtfNS}, the difference $u_\varepsilon - u_\delta$ belongs to
   $C (\overline{I},H) \cap L^2 (I,\cD_V)$
and vanishes at $t = 0$.
Moreover, on subtracting equality (\ref{eq.NSpowif}) for $u_\delta$ from that for $u_\varepsilon$
we obtain
$$
   (u_\varepsilon - u_\delta)'
 + \left( \varepsilon \iD^m u_{\varepsilon} - \delta \iD^m u_{\delta} \right)
 + \nu \iD (u_\varepsilon - u_\delta)
 + \left( N (u_{\varepsilon}) - N (u_{\delta}) \right)
 =
   0
$$
a.e. on $I$.

In order to find appropriate estimates for the difference $u_\varepsilon - u_\delta$ we take
the duality pairing of the latter equality with the special test section
   $v (t) = u_\varepsilon (t) - u_\delta (t)$,
and integrate it from $0$ to $t$ for arbitrary fixed $t \in I$.
A trivial verification leads to the equality
\begin{eqnarray*}
\lefteqn{
   \int_0^t ((u_\varepsilon - u_\delta)', v) dt'
 + \varepsilon \int_0^t (\iD^m (u_\varepsilon - u_\delta), v) dt'
 + \nu \int_0^t (\iD (u_\varepsilon - u_\delta), v) dt'
}
\\
 & = &
 -\,
   (\varepsilon - \delta) \int_0^t (\iD^m u_\delta, v) dt'
 - \int_0^t (N (u_\varepsilon) - N (u_\delta), v) dt'
   \hspace{2cm}
\end{eqnarray*}
which is valid for all $t \in I$.
Since
\begin{eqnarray*}
   \Re\, ((u_\varepsilon - u_\delta)', u_\varepsilon - u_\delta)
 & = &
   \frac{1}{2} \frac{\partial}{\partial t'}\, \| u_\varepsilon - u_\delta \|_H^2,
\\
   (\iD^m (u_\varepsilon - u_\delta), u_\varepsilon - u_\delta)
 & = &
   \| \iD^{m/2} (u_\varepsilon - u_\delta) \|_H^2,
\\
   (\iD^m u_\delta, u_\varepsilon - u_\delta)
 & = &
   (\iD^{m/2} u_\delta, \iD^{m/2} (u_\varepsilon - u_\delta))
\end{eqnarray*}
and
\begin{eqnarray*}
\lefteqn{
   \left( N (u_{\varepsilon}) - N (u_{\delta}), u_{\varepsilon} - u_{\delta} \right)
}
\\
 & = &
   \left( N (u_{\varepsilon} - u_{\delta}, u_{\varepsilon})
        + N (u_{\delta}, u_{\varepsilon})
        - N (u_{\delta}, u_{\delta} - u_{\varepsilon})
        - N (u_{\delta}, u_{\varepsilon}),
          u_{\varepsilon} - u_{\delta}
   \right)
\\
 & = &
   \left( N (u_{\varepsilon} - u_{\delta}, u_{\varepsilon})
        + N (u_{\delta}, u_{\varepsilon} - u_{\delta}), u_{\varepsilon} - u_{\delta}
   \right)
\\
 & = &
 -\,
   \left( N (u_{\varepsilon} - u_{\delta}), u_{\varepsilon} \right)
\end{eqnarray*}
a.e. on $I$,
   the last equality being a consequence of Corollary \ref{c.algestr},
we arrive at the equality
\begin{eqnarray}
\label{eq.coteieq}
\lefteqn{
   \frac{1}{2}\, \| u_\varepsilon (t) - u_\delta (t) \|_H^2
 + \varepsilon \int_0^t \| \iD^{m/2} (u_\varepsilon - u_\delta) \|_H^2 dt'
 + \nu \int_0^t \| \iD^{1/2} (u_\varepsilon - u_\delta) \|_H^2 dt'
}
\nonumber
\\
 \!\! & \!\! = \!\! & \!\!
 -\,
   (\varepsilon - \delta)\, \Re \int_0^t (\iD^{m/2} u_\delta, \iD^{m/2} (u_\varepsilon - u_\delta)) dt'
 + \Re \int_0^t \left( N (u_{\varepsilon} - u_{\delta}), u_{\varepsilon} \right) dt'
\nonumber
\\
\end{eqnarray}
for almost all $t \in I$.
Here, we write
$
   \| \iD^{1/2} v \|_H^2
 = \| Av \|_H^2 + \| A^\ast v \|_H^2
$
to shorten notation.

Using the inequality
$
   \displaystyle
   a b \leq \frac{c}{2} a^2 + \frac{1}{2c} b^2
$
with
$
   \displaystyle
   c = \frac{1}{2}
$
we estimate the first term on the right-hand side of (\ref{eq.coteieq}) by
\begin{eqnarray*}
\lefteqn{
   \Big|
 -\,
   (\varepsilon - \delta)\, \Re \int_0^t (\iD^{m/2} u_\delta, \iD^{m/2} (u_\varepsilon - u_\delta)) dt'
   \Big|
}
\\
 & \leq &
   \frac{\varepsilon-\delta}{4}
   \int_0^t \| \iD^{m/2} u_\delta \|_H^2 dt'
 + (\varepsilon-\delta)
   \int_0^t \| \iD^{m/2} (u_\varepsilon - u_\delta) \|_H^2 dt'
\end{eqnarray*}
for all $t \in I$.

In order to estimate the second term on the right-hand side of (\ref{eq.coteieq}), we apply
inequality (\ref{eq.utHiqwg}) with $q = 2$ and $s = \infty$.
Namely, there is a constant $c$ independent of $\delta$ and $\varepsilon$, such that
$$
   | (N (u_{\varepsilon} - u_{\delta}), u_{\varepsilon}) |
 \leq
   c\,
   \| u_{\varepsilon} - u_{\delta} \|_{L^2 (\cX,F^i)}
   \| u_{\varepsilon} - u_{\delta} \|_{H^1 (\cX,F^i)}
   \| u_{\varepsilon} \|_{L^\infty (\cX,F^i)}
$$
a.e. on $I$ for all $\delta$ and $\varepsilon$.
As the Dirichlet prolem for the Laplacian $\iD$ is elliptic, it follows that
$$
   \| u_{\varepsilon} - u_{\delta} \|_{H^1 (\cX,F^i)}^2
 \leq
   C\,
   \Big( \| \iD^{1/2} (u_{\varepsilon} - u_{\delta}) \|_{L^2 (\cX,F^i)}^2
       + \| u_{\varepsilon} - u_{\delta} \|_{L^2 (\cX,F^i)}^2 \Big)
$$
with $C$ a constant independent of $\delta$, $\varepsilon$ and $t \in I$.
Hence, on arguing as above we obtain
\begin{eqnarray*}
\lefteqn{
   \Big| \Re \int_0^t \left( N (u_{\varepsilon} - u_{\delta}), u_{\varepsilon} \right) dt'
   \Big|
}
\\
 & \leq &
   \frac{c}{2}
   \int_0^t
   \Big( \Big( \frac{1}{C'} \| u_\varepsilon - u_\delta \|_{H^1 (\cX,F^i)}^2 \Big)^2
       + \Big( C' \| u_\varepsilon \|_{L^\infty (\cX,F^i)} \| u_\varepsilon - u_\delta \|_H \Big)^2
   \Big)
   dt'
\\
 & \leq &
   \frac{\nu}{2}
   \int_0^t \| \iD^{1/2} (u_\varepsilon - u_\delta) \|_H^2 dt'
 + \int_0^t
   \Big( \frac{\nu}{2} + \frac{c}{2} \frac{C c}{\nu} \| u_\varepsilon \|_{L^\infty (\cX,F^i)}^2 \Big)
   \| u_\varepsilon - u_\delta \|_H^2
   dt',
\end{eqnarray*}
where
$
   \displaystyle
   (C')^2 = \frac{C c}{\nu}.
$

Substituting these estimates into (\ref{eq.coteieq}) yields
\begin{eqnarray*}
\lefteqn{
   \| u_\varepsilon (t) - u_\delta (t) \|_H^2
 + 2 \delta \int_0^t \| \iD^{m/2} (u_\varepsilon - u_\delta) \|_H^2 dt'
 + \nu \int_0^t \| \iD^{1/2} (u_\varepsilon - u_\delta) \|_H^2 dt'
}
\nonumber
\\
 & \leq &
   \frac{\varepsilon - \delta}{2}
   \int_0^t \| \iD^{m/2} u_\delta \|_H^2 dt'
 + \int_0^t
   \Big( \nu + c \frac{C c}{\nu} \| u_\varepsilon \|_{L^\infty (\cX,F^i)}^2 \Big)
   \| u_\varepsilon - u_\delta \|_H^2
   dt'
\end{eqnarray*}
for all $t \in I$.
We now apply the Gronwall inequality for continuous functions (\ref{eq.Gronwall}) to
obtain
$$
   \| u_\varepsilon (t) - u_\delta (t) \|_H^2
 \leq
   \frac{\varepsilon - \delta}{2}
   \int_0^t \| \iD^{m/2} u_\delta \|_H^2 dt'\,
   \exp
   \int_0^t \Big( \nu + c \frac{C c}{\nu} \| u_\varepsilon \|_{L^\infty (\cX,F^i)}^2 \Big)
   dt'
$$
for all $t \in I$.
By assumption, the family $u_\varepsilon$ is bounded in $L^2 (I,H^m (\cX,F^i))$ uniformly
in the parameter $\varepsilon$.
Therefore, there is a constant $B$ with the property that
\begin{eqnarray*}
   \int_0^t \| \iD^{m/2} u_\delta (t') \|_H^2 dt'
 & \leq &
   c' \int_0^T \| u_\delta (t') \|_{H^m (\cX,F^i)}^2 dt'
\\
 & \leq &
   B
\end{eqnarray*}
for all
   $t \in I$ and
   $\delta > 0$.
Furthermore, since $m > n/2$, we deduce from the Sobolev embedding theorem that
\begin{eqnarray*}
   \int_0^t \Big( \nu + c \frac{C c}{\nu} \| u_\varepsilon (t') \|_{L^\infty (\cX,F^i)}^2 \Big) dt'
 & \leq &
   \nu T + c' c \frac{C c}{\nu} \int_0^T \| u_\varepsilon (t') \|_{H^m (\cX,F^i)}^2 dt'
\\
 & \leq &
   A
\end{eqnarray*}
for all
   $t \in I$ and
   $\varepsilon > 0$,
where $A$ is a constant independent of $t$ and $\varepsilon$.
It follows that
\begin{eqnarray}
\label{eq.coteieqe}
   \sup_{t \in I} \| u_\varepsilon (t) - u_\delta (t) \|_H^2
 + \nu \int_0^T \| \iD^{1/2} (u_\varepsilon - u_\delta) \|_H^2 dt'
 & \leq &
   \frac{B}{2}\, \Big( 1 + A \exp A \Big)\,
   |\varepsilon - \delta|
\nonumber
\\
\end{eqnarray}
whenever
   $\delta, \varepsilon > 0$.

Estimate (\ref{eq.coteieqe}) shows that $u_{\varepsilon}$ is a uniformly continuous function
of parameter $\varepsilon > 0$ with values in both
   $L^\infty (I, L^2 (\cX,F^i))$ and
   $L^2 (I, H^1 (\cX,F^i))$.
Since these are Banach spaces, $u_\varepsilon$ converges in both the spaces to a section $u$
which lies in
   $L^\infty (I, L^2 (\cX,F^i))$ and
   $L^2 (I, H^1 (\cX,F^i))$,
as $\varepsilon \to 0$.
Moreover, the limit $u$ belongs to $L^2 (I,H^m (\cX,F^i))$, for
   the family $u_\varepsilon$ is bounded in $L^2 (I,H^m (\cX,F^i))$ uniformly in $\varepsilon$
and so
   it contains a subsequence $u_{\varepsilon_k}$ which converges in the weak topology of
   this space to $u$.

It remains to show that $u$ is a solution of the conventional Navier-Stokes equations with
initial data $u_0$ and zero Dirichlet data on the lateral boundary of $\cC$.
Since $\mathcal{S}$ lies dense in $S^1 \cap L^n (\cX,F^i)$, it suffices to prove the first
equality of (\ref{eq.conveNS}) for all $v \in \mathcal{S}$.

Let $v \in \mathcal{S}$.
Since $u_\varepsilon$ is a solution of the regularised Navier-Stokes equations, we get
\begin{eqnarray*}
\lefteqn{
   \Big( (u_\varepsilon (t), v) - (u_0, v) \Big)
 + \nu \int_0^t (\iD^{1/2} u_\varepsilon (t'), \iD^{1/2} v) dt'
}
\\
 & = &
   \int_0^t (f (t'), v) dt'
 - \varepsilon \int_0^t (\iD^{m/2} u_\varepsilon (t'), \iD^{m/2} v) dt'
 - \int_0^t (N (u_\varepsilon (t')), v) dt'
\end{eqnarray*}
for all $t \in I$.
The second summand on the right-hand side converges to zero, for the family $u_\varepsilon$
in bounded uniformly in $\varepsilon$ in the space $L^2 (I,H^m (\cX,F^i))$.
On the other hand,
$$
   \int_0^t (N (u_\varepsilon (t')), v) dt' \to \int_0^t (N (u (t')), v) dt'
$$
as $\varepsilon \to 0$, which is due to Lemma \ref{l.ptlfnlt}.
Hence, letting $\varepsilon \to 0$ in the above equality yields
\begin{equation}
\label{eq.yeati0t}
   \Big( \! (u (t), v) - (u_0, v) \! \Big)
 + \nu \! \int_0^t \! (\iD^{1/2} u (t'), \iD^{1/2} v) dt'
 + \! \int_0^t \! (N (u (t')), v) dt'
 = \! \int_0^t \! (f (t'), v) dt'
\end{equation}
for almost all $t \in I$.
Since the integrands in (\ref{eq.yeati0t}) are integrable functions of $t' \in I$, we may
differentiate this equality in $t$, thus achieving
$$
   (u' (t), v))
 + \nu\, (\iD^{1/2} u (t), \iD^{1/2} v)
 + (N (u (t)), v)
 = (f (t), v)
$$
for almost all $t \in I$ and all $v \in S^1$.
As $u \in L^\infty (I,H) \cap L^2 (I,S^1)$, we establish as in Section \ref{s.exiweak} that
$u$ coincides with a continuous function of $t \in [0,T]$ with values in $H$.
Hence, yet another consequence of (\ref{eq.yeati0t}) is that $u (0) = u_0$, and so $u$ is a
solution of (\ref{eq.conveNS}), as desired.
\end{proof}

As mentioned, problem (\ref{eq.NSweakf}) is a singular perturbation of the conventional
Navier-Stokes equations.
Hence, its solution need not converge to a solution of the conventional problem in any
strong topology even if this latter is unique.
For instance, the family
$
   u_\varepsilon (t) = (u_0 - 1) \exp (-t/\varepsilon) + 1
$
of solutions to the singularly perturbed initial problem
$$
\begin{array}{rclcl}
   \varepsilon u' + u
 & =
 & 1
 & \mbox{for $t \in I$},
\\
   u (0)
 & =
 & u_0
 &
\end{array}
$$
converges to the solution of the unperturbed equation $u = 1$ uniformly on $\overline{I}$
if and only if $u_0 = 1$.

\bigskip

\textit{Acknowledgments\,}
The authors are greatly indebted to Prof. A. Feldmeier for helpful comments concerning
the von Neumann-Richtmyer artificial viscosity in numerical simulations with second
order partial differential equations.
The first author gratefully acknowledges the financial support of
   the Russian Federation Government for scientific research under the supervision of
   leading scientist at the Siberian Federal University, contract N 14.Y26.31.0006,
and
   the Alexander von Humboldt Foundation.

\section*{Appendix}
\label{l.appendix}

The following lemma is needed to obtain the weak form of the singularly perturbed Navier-Stokes
equations satisfied by the weak solution.
This lemma is also used to describe the behaviour of the solution of the regularised equations
as $\varepsilon \to 0+$.

\bigskip

\noindent
\textbf{Lemma}
\textit{%
Suppose $m \geq (n+2)/4$.
Let $\{ u_k \}$ be a sequence of sections of $F^i$ over $\cC$ which converges
   in the weak* topology of $L^\infty (I,H)$,
   weakly in $L^2 (I,S)$ and
   strongly in $L^2 (I,H)$.
Then
\begin{eqnarray*}
   \lim_{k \to \infty} \int_0^T (u_k (t), v' (t))\, dt
 & = &
   \int_0^T (u (t), v' (t))\, dt,
\\
   \lim_{k \to \infty} \int_0^T (\iD^{m/2} u_k (t), \iD^{m/2} v (t))\, dt
 & = &
   \int_0^T (\iD^{m/2} u (t), \iD^{m/2} v (t))\, dt,
\\
   \lim_{k \to \infty} \int_0^T (N (u_k (t)), v (t))\, dt
 & = &
   \int_0^T (N (u (t)), v (t))\, dt
\end{eqnarray*}
for all sections $v \in C (\overline{I},S)$ satisfying $v' \in L^2 (I,H)$.
}

\bigskip

\begin{proof}
The first equality follows from the facts that
   $u_k \to u$ strongly in $L^2 (I,H)$ and
   $v' \in L^2 (I,H)$.

To prove the second equality we use the fact that $u_k \to u$ in the weak topology of
$L^2 (I,S)$.
That is,
$$
   \lim_{k \to \infty} \int_0^T (u_k - u, w) dt = 0
$$
for all $w \in L^2 (I,S')$.
On the other hand, the operator $\iD^m$ maps $S$ continuously into $S'$.
Hence it follows that
\begin{eqnarray*}
   \int_0^T (\iD^{m/2} u_k (t), \iD^{m/2} v (t))
 & = &
   \int_0^T (u_k (t), \iD^m v (t)) dt
\\
 & \to &
   \int_0^T (u (t), \iD^{m} v (t)) dt
\\
 & = &
   \int_0^T (\iD^{m/2} u (t), \iD^{m/2} v (t)) dt,
\end{eqnarray*}
as desired.

Finally, in order to show the last equality of the lemma we observe that the difference
$$
   \int_0^T (N (u_k (t)), v (t))\, dt - \int_0^T (N (u (t)), v (t))\, dt
$$
just amounts to
$$
   \int_0^T (N (u_k (t) - u (t), u_k (t)), v (t))\, dt
 + \int_0^T (N (u (t), u_k (t) - u (t)), v (t))\, dt.
$$
Consider the first integral and recall the estimate of the trilinear form given by
Lemma \ref{l.totnwst}.
Since $v \in C (\overline{I},S)$, we get
\begin{eqnarray*}
\lefteqn{
   \big| \int_0^T (N (u_k - u, u_k), v)\, dt \big|
}
\\
 & \leq &
   c
   \int_0^T
   \| u_k - u \|_H^{1/2} D (u_k - u)^{1/4} \| u_k \|_H^{1/2} D (u_k)^{1/4} D (v)^{1/2}
   dt
\\
 & \leq &
   c
   \Big( \int_0^T \| u_k - u \|_H D (u_k - u)^{1/2} dt \Big)^{1/2}
   \Big( \int_0^T \| u_k \|_H D (u_k)^{1/2} dt \Big)^{1/2}
\\
 & \leq &
   c
   \Big( \| u_k - u \|_{L^2 (I,H)} \| u_k - u \|_{L^2 (I,S)} \Big)^{1/2}
   \Big( \| u_k \|_{L^2 (I,H)} \| u_k \|_{L^2 (I,S)} \Big)^{1/2},
\end{eqnarray*}
where $c$ stands for a constant independent of $k$ and different in diverse applications.
Since
   $u_k \in L^2 (I,S)$ and
   $u_k \to u$ in the norm of $L^2 (I,H)$,
the right-hand side converges to zero, as $k \to \infty$.
Convergence arguments for the second integral are similar.
\end{proof}

\end{document}